\documentclass[english,notitlepage]{amsart}
\usepackage[latin9]{inputenc}
\usepackage{geometry}
\usepackage{esint}
\usepackage{hyperref}
\usepackage{dsfont}

\usepackage{amssymb,amsmath,amsthm,amscd}
\usepackage{mathrsfs}
\usepackage{tikz}
\usetikzlibrary{positioning}
\usepackage{graphics}
\usetikzlibrary{matrix,arrows}
\usepackage{amssymb} 
\usepackage{babel}
\usepackage{amstext}
\usepackage{amscd}   
\usepackage{epsfig}  
\usepackage{rotating}
\usepackage{psfrag}  
\usepackage{multicol}
\usepackage{multirow}
\usepackage{array}
\usepackage{verbatim}
\usepackage{enumitem}
\usepackage{colortbl}
\usepackage{hhline}
\usepackage{xcolor}
\usepackage[abs]{overpic}

\makeatletter
\numberwithin{equation}{section}
\numberwithin{figure}{section}

\theoremstyle{plain}
\newtheorem{thm}{Theorem}
  \theoremstyle{plain}
  \numberwithin{thm}{section}
  \newtheorem{cor}[thm]{Corollary}
  \theoremstyle{plain}
  \newtheorem{lem}[thm]{Lemma}
  \theoremstyle{remark}
  \newtheorem{rem}[thm]{Remark}
    \theoremstyle{remark}
  \newtheorem{example}[thm]{Example}
    \newtheorem{examples}[thm]{Examples}
   \theoremstyle{plain}
  
  \def\Ddots{\mathinner{\mkern1mu\raise\p@
\vbox{\kern7\p@\hbox{.}}\mkern2mu
\raise4\p@\hbox{.}\mkern2mu\raise7\p@\hbox{.}\mkern1mu}}
\makeatother

\usepackage{babel}

\newtheorem{definition}{Definition}
\numberwithin{definition}{section}

\newtheorem*{result}{Theorem}

\newcommand{\norm}[1]{\left\| #1 \right\|}
\newcommand{\mklm}[1]{\left\{ #1 \right\}}
\newcommand{\eklm}[1]{\left\langle #1 \right\rangle}

\renewcommand{\d}{\,d}
\newcommand{\N}{{\mathbb N}}
\newcommand{\Z}{{\mathbb Z}}
\newcommand{\C}{{\mathbb C}}
\newcommand{\Ccal}{{\mathcal C}}
\newcommand{\R}{{\mathbb R}}

\newcommand{\B}{{\mathcal B}}

\newcommand{\E}{{\mathcal E}}

\newcommand{\Jbb}{{\mathbb J }}

\renewcommand{\O}{{\mathcal O}}

\newcommand{\Xbf}{{\mathbf X}}

\newcommand{\W}{{\mathcal W}}
\newcommand{\Wh}{{{\mathcal W}_h}}

\renewcommand{\epsilon}{\varepsilon}

\renewcommand{\rho}{\varrho}

\newcommand{\bdm}{\begin{displaymath}}
\newcommand{\edm}{\end{displaymath}}
\newcommand{\bq}{\begin{equation}}
\newcommand{\eq}{\end{equation}}
\newcommand{\bqn}{\begin{equation*}}
\newcommand{\eqn}{\end{equation*}}

\newcommand{\Cinft}{{\rm C^{\infty}}}
\newcommand{\CT}{{\rm C^{\infty}_c}}

\newcommand{\Sob}{{\rm H}}
\renewcommand{\L}{{\rm L}}

\newcommand{\Ncal}{{\mathcal N}}

\renewcommand{\S}{{\mathcal S}}

\newcommand{\SO}{\mathrm{SO}}

\newcommand{\g}{{\bf \mathfrak g}}
\renewcommand{\k}{{\bf \mathfrak k}}
\renewcommand{\t}{{\bf \mathfrak t}}

\newcommand{\Ad}{\mathrm{Ad}\,}

\newcommand{\id}{\mathrm{id}\,}

\renewcommand{\det}{\mathrm{det}\,}

\newcommand{\vol}{\text{vol}\,}

\newcommand{\Crit}{\mathrm{Crit}}

\DeclareMathOperator{\supp}{supp\,}

\DeclareMathOperator{\tr}{tr}
\DeclareMathOperator{\gd}{\partial}

\begin{document}
\author{Benjamin K\"uster}
\address{Philipps-Universit\"at Marburg, Fachbereich Mathematik und Informatik, Hans-Meerwein-Str., 35032 Marburg, Germany}
\email[Benjamin K\"uster]{\href{mailto:bkuester@mathematik.uni-marburg.de}{bkuester@mathematik.uni-marburg.de}}
\author{Pablo Ramacher}
\email[Pablo Ramacher]{\href{mailto:ramacher@mathematik.uni-marburg.de}{ramacher@mathematik.uni-marburg.de}}
\title[Semiclassical analysis and symmetry reduction I]{Semiclassical analysis and symmetry reduction I. \\ 
Equivariant Weyl law  for invariant Schr\"odinger operators on compact  manifolds} 
\date{\today}

\begin{abstract}
We study  the spectral properties of Schr\"odinger operators on a compact connected Riemannian manifold $M$ without boundary in case that the underlying Hamiltonian system possesses certain symmetries. More precisely, if $M$ carries an isometric and effective action of a compact connected Lie group $G$, we  prove a generalized equivariant version of the semiclassical Weyl law with an estimate for the remainder, using a semiclassical functional calculus for $h$-dependent functions and relying on recent results on singular equivariant asymptotics.  These results will be used to derive an equivariant quantum ergodicity theorem in Part II of this work.  When $G$ is trivial, one recovers the classical results. 
\end{abstract}

\maketitle

\setcounter{tocdepth}{1}
\tableofcontents{}

\section{Introduction}

\subsection{Problem and setup}
Let $M$ be a compact connected boundary-less  Riemannian $\Cinft$-manifold of dimension $n$ with Riemannian volume density $dM$, and denote by $\Delta$ the Laplace-Beltrami operator on $M$ with domain $\Cinft(M)$. 
One of the central problems in spectral geometry is to study the properties of eigenvalues and eigenfunctions of $-\Delta$ in the limit of large eigenvalues. Concretely, let $\mklm{u_j}$ be an orthonormal basis of $\L^2(M)$ of eigenfunctions of $-\Delta$ with respective eigenvalues $\mklm{E_j}$, repeated according to their multiplicity. As $E_j \to \infty$,  one is interested among other things in the asymptotic distribution of eigenvalues,  the pointwise convergence of the $u_j$, bounds of the $\L^p$-norms of the $u_j$ for $1 \leq p \leq \infty$, and the weak convergence of the measures $|u_j|^2 dM$. 
The study of the  distribution of eigenvalues has a  long history  that goes back to  Weyl \cite{weyl},  Levitan \cite{levitan52}, Avacumovi\v{c} \cite{avacumovic}, and H\"ormander \cite{hoermander68},  
while the behavior of eigenfunctions has attracted much attention in more recent times, one of the major results  in this direction  being the \emph{quantum ergodicity theorem} for  chaotic systems, due to  Shnirelman \cite{shnirelman}, Zelditch \cite{zelditch1987}, and Colin de Verdi\`{e}re \cite{colindv}. 
This paper is the first in a sequel which  addresses these problems for Schr\"odinger operators in case that the underlying classical system possesses certain symmetries. 

In this first part, we shall concentrate  on the distribution of eigenvalues. The question is then how the symmetries of the underlying Hamiltonian system  determine the fine structure of the spectrum in accordance with the correspondence principle of quantum physics. 
To explain things more precisely, let $G$ be a compact connected Lie group that acts effectively and isometrically on $M$. Note that  there might be   orbits  of different dimensions, and that the orbit space $\widetilde M:=M/G$  won't be a manifold in general, but a topological quotient space.  If $G$ acts on $M$ with finite isotropy groups, $\widetilde M$ is a compact orbifold, and its  singularities are not too severe. Consider now a \emph{Schr\"odinger operator} on $M$
\bqn
\breve{P}(h)=-h^2\Delta + V,\qquad \breve{P}(h):\Cinft(M)\to \Cinft(M),\qquad h\in(0,1],
\eqn
$V\in \Cinft(M,\R)$ being a $G$-invariant potential. $\breve{P}(h)$ has a unique self-adjoint extension 
\bq
\label{eq:13.08.15}
P(h):\mathrm{H}^2(M)\to \L^2(M)
\eq
as an unbounded operator in $\L^2(M)$, where $\mathrm{H}^2(M)\subset \L^2(M)$ denotes the second Sobolev space, and one calls $P(h)$ a Schr\"odinger operator, too. The quantum mechanical properties of $P(h)$ can be described by studying its spectrum. For each $h\in(0,1]$, it is discrete, consisting of  eigenvalues $\{E_j(h)\}_{j\in \N}$ which we repeat according to their multiplicity and which form a non-decreasing sequence unbounded towards $+\infty$.  The associated sequence of eigenfunctions $\{u_{j}(h)\}_{j\in\mathbb{N}}$ constitutes a Hilbert basis in $\L^2(M)$ of smooth functions,  and the eigenspaces are finite-dimensional. When studying the spectral asymptotics of Schr\"odinger operators, one often uses the \emph{semiclassical method}. Instead of examining the spectral properties of $P(h)$ for fixed $h=\hbar$ and high energies, $\hbar$ being Planck's constant, one considers fixed energy intervals, allowing $h\in(0,1]$ to become small. The two methods are essentially equivalent. In the special case $V\equiv0$, the Schr\"odinger operator is just a rescaled version of $-\Delta$ so that the semiclassical method can be used to study the spectral asymptotics of the Laplace-Beltrami operator. Now, since  $P(h)$ commutes with the isometric $G$-action, one can use representation theory to describe  the spectrum of $P(h)$ in a more refined way. Indeed, by the Peter-Weyl theorem, the unitary left-regular representation of $G$ 
\begin{equation*}
G\times \L^2(M)  \to  \L^2(M),\qquad (g,f)  \mapsto  \left(L_gf:\, x\mapsto f(g^{-1}\cdot x)\right),
\end{equation*}
has an orthogonal decomposition into isotypic components according to
\bq
\label{eq:PW} 
\L^2(M)=\bigoplus_{\chi \in \widehat G} \L^2_\chi(M),\qquad \L^2_\chi(M)= T_\chi \, \L^2(M),
\eq
where  $\widehat{G}$ denotes the set of isomorphism classes of irreducible unitary $G$-representations which  we identify with the character ring of $G$, while $T_\chi:\L^2(M)\to \L^2_\chi(M)$
are the  associated orthogonal projections  given by 
\bq
T_{\chi}: f  \mapsto  \Big(x\mapsto d_{\chi}\int_{G}\overline{\chi(g)}f(g^{-1}\cdot x)\, dg\Big),\label{eq:groupproj}
\eq
where $dg$ is the normalized Haar measure on $G$, and $d_\chi$ the dimension of an irreducible representation of class $\chi$.  Since each eigenspace of the Schr\"odinger operator $P(h)$ constitutes a unitary $G$-module, it has an analogous decomposition into a direct sum of irreducible $G$-representations, which represents the so-called \emph{fine structure} of the spectrum of $P(h)$. To study this fine structure asymptotically, consider for a fixed $\chi\in \widehat G$ and any operator
$A:D\to \L^2(M)$  on a $G$-invariant subset $D\subset \L^2(M)$ the corresponding \emph{reduced operator}  \[
A_{\chi}:=T_{\chi}\circ A\circ T_{\chi}|_D.\]
 Since $P(h)$ commutes with $T_\chi$, the reduced operator $P(h)_\chi$ corresponds to the bi-restriction $P(h)|_\chi:\L^2_\chi(M)\cap \Sob^2(M)\to \L^2_\chi(M)$.  More generally, one can consider  bi-restrictions of $P(h)$ to $h$-dependent sums of isotypic components of the form
\[
\L^2_\Wh(M)=\bigoplus_{\chi \in \Wh} \L^2_\chi(M),
\]
where the  $\Wh\subset \widehat G$ are  appropriate finite subsets  whose cardinality is allowed to grow in a controlled way as $h\to 0$. A natural problem is  then to examine the spectral asymptotics of $P(h)$ bi-restricted to $\L^2_\Wh(M)$. The study of a single isotypic component corresponds to choosing $\Wh=\{\chi\}$ for all $h$ and a fixed $\chi\in \widehat G$.  Note that, so far, it is  irrelevant whether the group action has various different orbit types or not. 

On the other hand, the principal symbol of the Schr\"odinger operator \eqref{eq:13.08.15} is given by the $G$-invariant symbol function 
\bq
\label{eq:13.08.15a}
p: T^*M\to \R,\qquad(x,\xi)\mapsto \norm{\xi}^2_x + V(x),
\eq
and describes the classical mechanical properties of the underlying Hamiltonian system on the co-tangent  bundle $T^*M$ of $M$ with   canonical symplectic form $\omega$.
Consider now for a regular value $c$ of  $p$  the hypersurface  $\Sigma_c:=p^{-1}(\{c\})\subset T^*M$, which 
carries a canonical hypersurface measure $d\Sigma_c$  induced by $\omega$. In the special case that  $\Sigma_c=S^\ast M$ is the co-sphere bundle, $d\Sigma_c=d(S^*M)$ is commonly called the \emph{Liouville measure}. 
To describe the classical dynamical properties of the system, it is  convenient to divide out the symmetries, which can be done by  performing a procedure called \emph{symplectic reduction}. 
Namely, let $\Jbb:T^*M\rightarrow \g^\ast$ denote the momentum map of the induced Hamiltonian $G$-action on $T^*M$, which represents the conserved quantitites of the dynamical system, and consider  the topological quotient space 
\bqn
\widetilde \Omega:=\Omega/G, \qquad   \Omega:=\Jbb^{-1}(\mklm{0}).
\eqn
 In contrast to the situation encountered in the Peter-Weyl theorem,  the orbit structure of the underlying $G$-action on $M$ is not at all irrelevant for the symplectic reduction. Indeed, if the $G$-action is not free the spaces $\Omega$ and $\widetilde \Omega$ need not be manifolds. Nevertheless, they are stratified spaces, where each stratum is a smooth manifold that consists of  orbits of one particular type. In particular, $\Omega$ and $\widetilde \Omega$ each have a \emph{principal stratum} $\Omega_\text{reg}$ and $\widetilde \Omega_\text{reg}$, respectively, which is the smooth manifold consisting of (the union of) all orbits of type $(H)$, where $H$ denotes a principal  isotropy group of the $G$-action on $M$. Moreover, $\widetilde \Omega_\text{reg}$ carries a canonical symplectic structure with an associated volume density.  The intersection $\Omega_\text{reg}\cap\Sigma_c$ is transversal, so that it defines a hypersurface in $\Omega_\text{reg}$. Similarly, when passing to the orbit space, we obtain the hypersurface $\widetilde \Sigma_c=\big(\Omega_\text{reg}\cap\Sigma_c\big)/G\subset \widetilde \Omega_\text{reg}$ that carries a canonical measure $d\widetilde{\Sigma}_c$ induced from the volume density on $\widetilde \Omega_\text{reg}$, and one can interpret the  measure space $(\widetilde{\Sigma}_c,d\widetilde{\Sigma}_c )$ as the symplectic reduction of the measure space $(\Sigma_c,d\Sigma_c)$. For a detailed exposition of these facts, the reader is referred to Sections \ref{sec:2.3} and \ref{subsec:spaces}.\\

Let us now come  back to our initial question. In what follows, we shall study the distribution of the eigenvalues of $P(h)$ along $h$-dependent families of isotypic components $\L^2_{\Wh}(M)$ in the Peter-Weyl decomposition of $\L^2(M)$ as $h \to 0$, and the way their distribution  in a spectral window $[c,c+h^\delta]$ around a regular value $c$ of $p$ is related to the symplectic reduction $(\widetilde{\Sigma}_c,d\widetilde{\Sigma}_c )$ of the corresponding Hamiltonian system, $\delta>0$ being a suitable small number. Similar problems were studied for $h$-pseudodifferential operators in  $\R^n$ in \cite{helffer-elhouakmi91}, \cite{cassanas}, \cite{weich}, and within the classical high-energy approach in \cite{donnelly78}, \cite{bruening-heintze79}, \cite{helffer-robert84,helffer-robert86}, and \cite{ramacher10}. In our approach, we shall combine well-known methods from semiclassical analysis and symplectic reduction with results on singular equivariant asymptotics recently developed in \cite{ramacher10, ramacher15a}.
\subsection{Results}
\label{sec:13.08.15}
To describe our results, we need to fix some additional notation.
We denote by $\kappa$ the dimension of the $G$-orbits in $M$ of principal type $(H)$, which agrees with the maximal dimension of a $G$-orbit in $M$, and assume throughout the whole paper that $\kappa<n=\dim M$. Furthermore, we denote by $\Lambda^{}$ the maximal number of elements of a totally ordered subset of isotropy types of the $G$-action on $M$. For an equivalence class  $\chi\in \widehat G$ with corresponding irreducible $G$-representation $\pi_\chi$, we write $[\pi_{\chi}|_{H}:\mathds{1}]$ for the multiplicity of the trivial representation in the restriction of  $\pi_\chi$ to $H$. Let $\widehat G'\subset \widehat G$ denote the subset consisting of those classes of representations that appear in the decomposition \eqref{eq:PW} of $\L^2(M)$.  In order to consider a growing number of isotypic components of $\L^2(M)$ in the semiclassical limit, we make the following
\begin{definition}
\label{def:semiclfam}
A family $\{\W_h\}_{h\in (0,1]}$ of finite sets $\W_h\subset \widehat G'$ is called \emph{semiclassical character family} if there exists a $\vartheta\geq 0$ such that for each $N\in \{0,1,2,\ldots\}$ and each differential operator $D$ on $G$ of order $N$, there is a constant $C>0$ independent of $h$ with
\[
\frac 1 {\# \W_h} \sum_{\chi\in \Wh} \frac{\norm{D\, \overline{\chi}}_\infty}{\left[ \pi_{\chi}|_{H}:\mathds{1}\right]}\leq C\,h^{-\vartheta N}\qquad \forall \;h\in(0,1].
\]
We call the smallest possible $\vartheta$ the \emph{growth rate} of the semiclassical character family.
\end{definition}

\begin{rem}
Note that $\left[ \pi_{\chi}|_{H}:\mathds{1}\right]\geq 1$ for $\chi \in \widehat G'$, since  the irreducible $G$-representations appearing in the decomposition \eqref{eq:PW} of $\L^2(M)$ are  precisely those representations appearing in $\L^2(G/H)$, and  by the Frobenius reciprocity theorem one has $\left[ \pi_{\chi}|_{H}:\mathds{1}\right]=[\L^2(G/H):\pi_\chi]$, compare \cite[Section 2]{donnelly78}.
\end{rem}

\begin{example}
\label{ex:SO(2)}
For $G=\SO(2)\cong S^1\subset \C$, one has $\widehat G\equiv\{\chi_k:k\in \Z\}$, where the $k$-th character $\chi_k:G\to \C$ is given by $\chi_k\big(e^{i\varphi}\big)=e^{ik\varphi}$. One then obtains a semiclassical character family with growth rate less or equal to $\vartheta$ by setting
$\Wh:=\{\chi_k:|k|\leq h^{-\vartheta}\}$.  
\end{example}

\begin{example}
More generally, let $G$ be a connected  compact semi-simple Lie group with Lie algebra $\g$ and $T\subset G$ a maximal torus with Lie algebra $\t$. By the Cartan-Weyl classification of irreducible finite dimensional representations of reductive Lie algebras over $\C$,   $\widehat G$ can be identified with the set of  dominant integral and $T$-integral linear forms $\Lambda$ on the complexification $\t_\C$ of the Lie algebra $\t$ of $T$. Let therefore denote $\Lambda_\chi\in \t_\C^\ast$ the element associated with a class $\chi \in \widehat G$, and put  $\mathcal{W}_h:=\mklm{\chi\in \widehat G: \, |\Lambda_\chi|\leq h^{-\vartheta}}$, where $\vartheta \geq 0$, $h \in (0, 1]$. Then  $\{\Wh\}_{h\in(0,1]}$ constitutes a  semiclassical character family  with growth rate less or equal to $\vartheta$ in the sense of Definition \ref{def:semiclfam},  see \cite[Section 3.2]{ramacher15a} for details.
\end{example}

Denote by $\Psi_{h,\delta}^{m}(M)$, $m \in \R\cup\{-\infty\}$, $\delta\in[0,\frac{1}{2})$, the set of semiclassical pseudodifferential operators on $M$ of order $(m,\delta)$. The principal symbols of these operators are represented by symbol functions in the classes $S^m_\delta(M)$, where the index $\delta$ describes the growth properties of the symbol functions as $h\to 0$, see Section \ref{subsec:semiclassical} for the precise definitions. An important point to note is that elements of $S^0_\delta(M)$ define operators on $\L^2(M)$ with operator norm bounded uniformly in $h$. Finally, for any measurable function $f$ with domain $D$ a $G$-invariant subset of $M$ or $T^*M$, we write
\bq
{\eklm{f}}_G(x):= \int_Gf(g\cdot x)\d g \label{eq:orbitalintegral},
\eq
and denote by $\widetilde{\eklm{f}}_G$ the function induced on the orbit space $D/G$ by the $G$-invariant function ${\eklm{f}}_G$. The main result of    this paper is the following
 
\begin{result}[{\bf Generalized equivariant semiclassical Weyl law}, Theorem \ref{thm:weyl2}] \label{res:1}  
Let $\delta\in \big(0,\frac{1}{2\kappa+4}\big)$ and choose an operator $B\in \Psi_{h,\delta}^0(M)\subset \B(\L^2(M))$ with principal symbol represented by $b\in S_\delta^0(M)$ and a semiclassical character family $\{\Wh\}_{h\in(0,1]}$ with growth rate $\vartheta<\frac{1-(2\kappa+4)\delta}{2\kappa+3}$.  Write
\[
J(h):=\big\{j\in \N:E_j(h)\in[c,c+h^\delta],\; \chi_j(h) \in \Wh\big\},
\]
where $\chi_j(h)\in \widehat G$ is defined by $u_j(h)\in \L^2_{\chi_j(h)}(M)$. Then, one has in the semiclassical limit $h\to 0$ \begin{align}
\label{eq:intBres}
\begin{split}
\frac{(2\pi)^{n-\kappa} h^{n-\kappa-\delta}}{\#\Wh}\sum_{J(h)}\frac{\langle Bu_{j}(h),u_{j}(h)\rangle_{\L^2(M)}}{d_{\chi_j(h)}\,[ \pi_{\chi_j(h)}|_{H}:\mathds{1}]}&=\intop_{{\Sigma}_c\cap \,\Omega_{\text{reg}}}b \, \frac {\d{\mu}_c}{\vol_\O}\\
 &+\; \mathrm{O}\Big(h^{\delta}+h^{\frac{1-(2\kappa+3)\vartheta}{2\kappa +4}-\delta}\left (\log h^{-1}\right)^{{\Lambda^{}}-1}\Big).
\end{split}
\end{align}
\end{result}

When considering only a fixed isotypic component $\L_{\chi}^2(M)$ the statement becomes simpler, yielding the asymptotic formula
\begin{align*}
\begin{split}
(2\pi)^{n-\kappa} h^{n-\kappa-\delta}\hspace*{-1.75em}\sum_{\begin{array}{c}\scriptstyle 
j\in\mathbb{N}:\,  u_{j}(h)\in \L_{\chi}^{2}(M),\\
\scriptstyle E_{j}(h)\in [c,c+h^\delta]\end{array}}\hspace*{-1.5em}&\langle Bu_{j}(h),u_{j}(h)\rangle_{\L^2(M)}=d_{\chi}\,[ \pi_{\chi}|_{H}:\mathds{1}]\int_{\widetilde{\Sigma}_c}\widetilde {\eklm{b}}_G\, \d\widetilde{\Sigma}_c \\
 &+\; \mathrm{O}\Big(h^{\delta}+h^{\frac{1}{2\kappa+4}-\delta}\left (\log h^{-1}\right)^{\Lambda^{}-1}\Big),
\end{split}
\end{align*}
see Theorem \ref{thm:weyl31}. Via co-tangent bundle reduction, the integral in the leading term of \eqref{eq:intBres} can actually be viewed as an integral over the smooth bundle 
\bqn
S^\ast_{\widetilde p,c}(\widetilde M_\text{reg}):=\big\{(x,\xi) \in T^\ast (\widetilde M_\text{reg}): \widetilde p(x,\xi)=c\big\},
\eqn
where $\widetilde M_\text{reg}$ is the space of principal orbits in $M$ and $\widetilde p$ is the function induced by $p$ on $T^\ast (\widetilde M_\text{reg})$, compare Lemma \ref{lem:isomorphic}.  In case that $\widetilde M$ is an orbifold, the mentioned integral is given by an integral over the orbifold bundle $S^\ast_{\widetilde p,c}(\widetilde M):=\big\{(x,\xi) \in T^\ast \widetilde M: \widetilde p(x,\xi)=c\big\}$, compare Remark \ref{rem:3.2}.

We prove Theorem \ref{thm:weyl2}  by applying a semiclassical trace formula for $h$-dependent test functions, which is the content of  Theorem \ref{thm:weakweyl}. The trace formula established here is a generalization of  \cite[Theorem 3.1]{weich} for Schr\"odinger operators to compact $G$-manifolds, $h$-dependent test functions, and $h$-dependent families of isotypic components. In particular, the fact that we consider $h$-dependent test functions  is crucial for the applications in Part II of this work \cite{kuester-ramacher15b}, and involves dealing with considerable technical difficulties. The proof of Theorem \ref{thm:weakweyl} relies on recent work of K\"uster \cite{kuester15} in which a semiclassical functional calculus for $h$-dependent functions is developed. Ultimately, the proof of Theorem \ref{thm:weakweyl} reduces to the asymptotic description of certain oscillatory integrals  that have recently been studied in  \cite{ramacher10,ramacher15a} by Ramacher using resolution of singularities. The involved phase functions are given  in terms of the underlying $G$-action on $M$, and if singular orbits occur,  the corresponding  critical sets are no longer smooth, so that a partial  desingularization process has to be implemented  in order to obtain asymptotics with remainder estimates via  the stationary phase principle. The explicit remainder estimates obtained in \cite{ramacher15a} do not only account for the quantitative form of the remainder in (\ref{eq:intBres}). They also provide the qualitative basis for our study of $h$-dependent families of isotypic components and the localization to the hypersurface $\widetilde \Sigma_c$. Without  the remainder estimates from  \cite{ramacher15a}, only a fixed single isotypic component $\L^2_\chi(M)$ could be studied, and only eigenvalues $E_j(h)$ lying in a non-shrinking energy strip of the form $[c,c+\varepsilon]$ with a fixed $\varepsilon>0$ could be considered, compare Remark \ref{rem:4.2}.

\begin{figure}[h!]
\hspace{-1.2em}\begin{minipage}{0.50\linewidth}
{\includegraphics[scale=0.8]{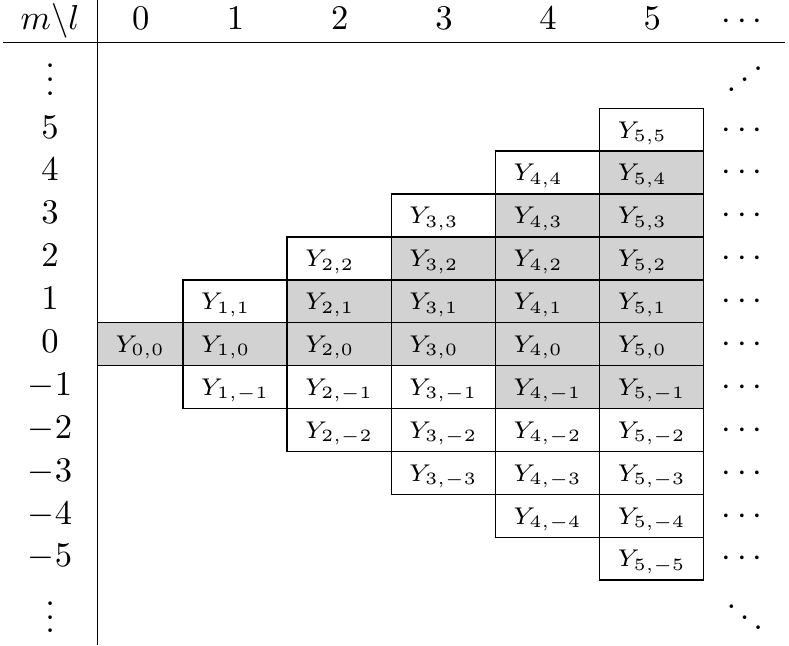}}
\centering\caption{A single isotypic component in the case $M=S^2$, $G=\SO(2)$. The $k$-th row spans the isotypic component $\L^2_{\chi_k}(S^2)$ and the $l$-th column represents the $l$-th eigenspace of $-\Delta$. \label{fig:ourapproach1}}
\end{minipage}\centering
\hspace{.5em} \begin{minipage}{0.50\linewidth}
\begin{overpic}[scale=0.8]{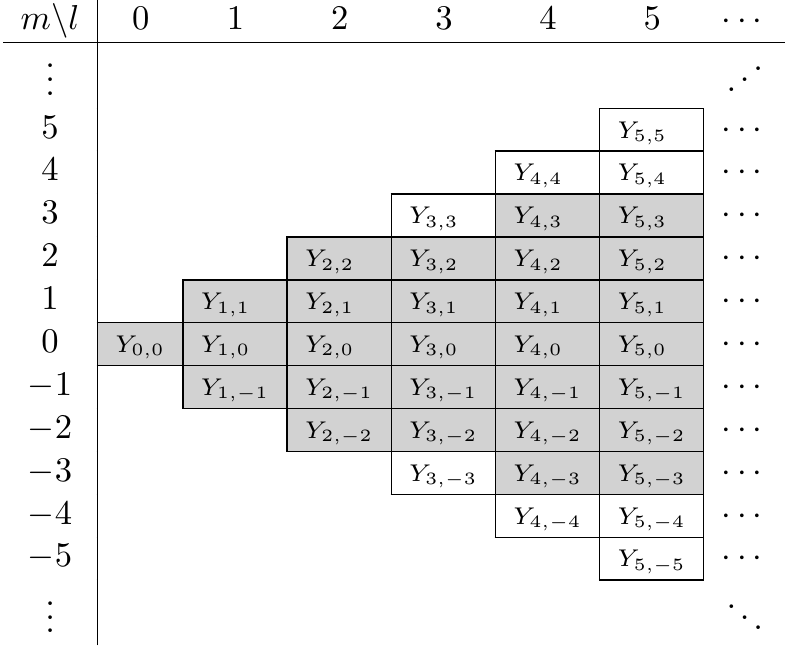}
\put(21.8,30,7){\includegraphics[scale=0.55]{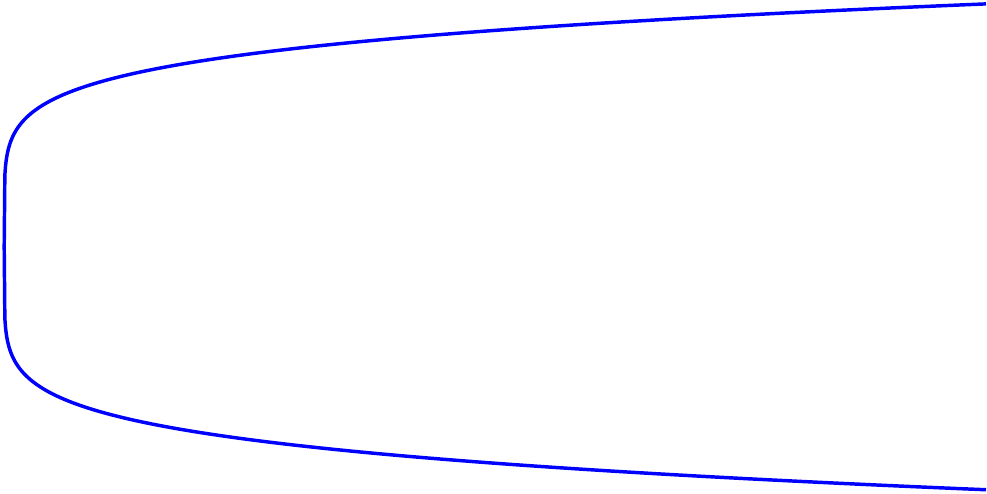}}
\end{overpic}
\centering
\caption{Spherical harmonics on $S^2$ in  isotypic components corresponding to a semiclassical character family with growth rate $\frac{1}{6}$. \label{fig:ourapproach2}}
\end{minipage}

\end{figure}

To illustrate our results, consider the standard $2$-sphere $S^2\subset \R^3$, acted upon by the group $\SO(2)$ of rotations around the $z$-axis in $\R^3$. This action has exactly two fixed points given by the north pole and the south pole of $S^2$. The orbits of all other points are circles, so  in this case we have $\kappa=1$ and $H=\{e\}$, the trivial group. The eigenvalues of $-\Delta$ on $S^2$ are given by the numbers $l(l+1)$, $ l=0,1,2,3\dots$, and the corresponding eigenspaces $\E_l$ are of dimension $2l+1$. They are  spanned by the spherical harmonics, given in spherical polar coordinates by
\begin{equation}
Y_{l,m}(\phi,\theta)={\scriptstyle\sqrt{\frac{2l+1}{4\pi}\frac{(l-m)!}{(l+m)!}}} P_{l,m}(\cos \theta)e^{im\phi}, \qquad 0\leq \phi<2\pi, \, 0 \leq \theta < \pi, \label{eq:Ylm}
\end{equation}
where $m \in \Z$, $|m|\leq l$, and  $P_{l,m}$ are the associated Legendre polynomials. Each subspace $\C \cdot Y_{lm}$ corresponds to an irreducible representation of $\SO(2)$, and each irreducible representation with character $\chi_k$  and $|k|\leq l$ occurs in the eigenspace $\E_l$ with multiplicity $1$, where $\chi_k$ is  as in Example \ref{ex:SO(2)}. Considering the limit $l\to \infty$ and the Laplacian $-\Delta$ is equivalent to studying  the limit $h\to 0$ and the semiclassical Laplacian $-h^2\Delta$. Figure \ref{fig:ourapproach1} depicts a single isotypic component, corresponding to a constant semiclassical character family. It means that one keeps $m$ fixed and studies the limit $l\to \infty$. On the other hand, the semiclassical character family from Example \ref{ex:SO(2)} for $\vartheta=\frac{1}{6}$  corresponds to the gray region in Figure \ref{fig:ourapproach2}. 

\subsection{Previously known results and outlook}
For general isometric and effective group actions the asymptotic distribution of the spectrum of an invariant operator along single isotypic components of $\L^2(M)$ was studied within the classical framework via heat kernel methods by Donnelly \cite{donnelly78} and Br\"uning--Heintze \cite{bruening-heintze79}. These methods allow to determine the leading term in the expansion, while  remainder estimates or growing families of isotypic components are  not accessible via this approach. On the other hand, using Fourier integral operator techniques, remainder estimates were obtained   for actions with orbits of the same dimension by  Donnelly \cite{donnelly78}, Br\"uning--Heintze \cite{bruening-heintze79}, Br\"uning \cite{bruening83}, Helffer--Robert \cite{helffer-robert84, helffer-robert86}, and Guillemin--Uribe \cite{guillemin-uribe90} in the classical setting, and El-Houakmi--Helffer \cite{helffer-elhouakmi91} and Cassanas \cite{cassanas}  in the semiclassical setting.  For general effective group actions, remainder estimates were derived by  Cassanas--Ramacher \cite{cassanas-ramacher09} and Ramacher \cite{ramacher10} in the classical, and by  Weich \cite{weich} in the semiclassical framework using resolution of singularites. The  idea of considering families of isotypic components that vary with the asymptotic parameter has been known since the 1980s. Thus, for  the Laplacian and free isometric group actions, so-called \emph{ladder subspaces} of $\L^2(M)$ and \emph{fuzzy ladders} have been considered in \cite{schrader, guillemin-uribe90} and \cite{Zelditch19921}, respectively. 
 \

Besides, there has  been much work in recent times  concerning the spectral theory of elliptic operators on orbifolds. Such spaces are locally homeomorphic to a quotient of Euclidean space by a finite group while, globally, any (reduced) orbifold is a quotient of a smooth manifold by a compact Lie group action with finite isotropy groups, that is, in particular, with no singular isotropy types \cite{adem-leida-ruan, moerdijk-mrcun}. As it turns out, the theory of elliptic operators on orbifolds is then essentially equivalent to the theory of invariant elliptic operators on manifolds carrying the action of a compact Lie group with finite isotropy groups \cite{bucicovschi, dryden-et-al, stanhope-uribe, kordyukov12}. 

Thus, in all the previously examined cases, except in \cite{cassanas-ramacher09}, \cite{ramacher10}, and \cite{weich}, no singular orbits occur and our work can be viewed as part of an attempt to develop a  spectral theory of elliptic operators on  general singular $G$-spaces. 

In Part II of this work \cite{kuester-ramacher15b} we will rely on Theorem \ref{thm:weyl2} to describe the ergodic properties of eigenfunctions of Schr\"odinger operators belonging to $h$-dependent families of isotypic components of $\L^2(M)$ in case that the symmetry-reduced Hamiltonian flow 
is ergodic, and prove a corresponding equivariant quantum ergodicity theorem. As further lines of research, it would be interesting to see whether our results can be generalized to $G$-vector bundles, as well as manifolds with boundary and non-compact situations. Furthermore, it might be fruitful to study the possible $h$-dependent families of irreducible representations from a more conceptional point of view.
Note that the non-equivariant semiclassical Weyl law suggests that for a given non-trivial $G$-action on $M$ there should be a critical growth rate $\vartheta_0>0 $ beyond which there  is not only a quantitatively, but also a qualitatively different asymptotic behavior.

\section{Background}\label{sec:prelim}

In this section we describe the setup from the introduction in more detail.

\subsection{Semiclassical analysis}\label{subsec:semiclassical}
In what follows, we shall give a brief overview of the theory of semiclassical symbol classes and pseudodifferential operators on smooth manifolds. For a detailed introduction, we refer the reader to \cite[Chapters 9
and 14]{zworski} and \cite[Chapters 7 and 8]{dimassi-sjoestrand}.  Semiclassical analysis developed out of the theory of pseudodifferential operators, a thorough exposition of which can be found in \cite{shubin}. 
An important feature that distinguishes semiclassical analysis from usual pseudodifferential operator theory is that instead of the usual symbol functions and corresponding operators, one considers families of symbol functions and pseudodifferential operators indexed by a global parameter \[h\in (0,1].\]
Essentially, the definitions of those families are obtained from the usual definitions by substituting in the symbol functions the cotangent variable $\xi$ by $h\xi$. To begin, we introduce semiclassical symbol classes which are invariant under pullbacks along diffeomorphisms. For $m\in \R$ and $\delta\in[0,\frac{1}{2})$,  one sets
\begin{align}
\nonumber S_\delta^{m}(\R^{n}):=&\Big\{a:\R^{2n}\times(0,1]\to \C: \; a(\cdot,h)\in \Cinft(\R^{2n})\;\forall\; h\in (0,1]\;\text{ and for all multiindices } s,t\;\\ &\quad \exists\; C_{s,t}>0:  |\partial^s_x\partial^t_\xi a(x,\xi,h)|\leq
C_{s,t}\left<\xi\right>^{m-|t|}h^{-\delta(|s|+|t|)}\;\forall\,x\in \R^n,\;h\in (0,1]\Big\}.\label{eq:kohnnirenberg}
\end{align}
The symbol classes (\ref{eq:kohnnirenberg}) generalize the classical Kohn-Nirenberg classes. In the literature one usually encounters only the case $\delta=0$. In our context\footnote{See \cite{dyatlov} for more applications of the symbol classes (\ref{eq:kohnnirenberg}).} it is natural to allow $\delta>0$, since we want to consider symbol functions that localize at points or other sets of low dimension as $h\to 0$. One can associate to each symbol function $s\in S_\delta^{m}(\R^{n})$  a continuous linear operator
\[
\text{Op}_h(s): \mathcal{S}(\R^n)\to \mathcal{S}(\R^n),\qquad f\mapsto \text{Op}_h(s)(f),
\]
on the Schwartz space $\S(\R^n)$ of rapidly decreasing functions on $\R^n$ given by
\bq
\text{Op}_h(s)(f)(x):=\frac 1 {(2\pi h)^n} \int_{\R^n} \int_{\R^n} e^{\frac i h (x-y)\cdot\eta}
s\Big(\frac{x+y}{2},\eta,h\Big) f(y) \d y \d \eta.\label{eq:defweyl}
\eq
Here $d\eta$ and $dy$ denote Lebesgue measure in $\R^n$.  This so-called \emph{Weyl-quantization} is motivated by the fact that the classical Hamiltonian $H(x,\xi)=\xi^2$ should correspond to the quantum Laplacian $-h^2 \Delta$, and real-valued symbol functions  to symmetric or  essentially self-adjoint operators.\\
Let now $M$ be a smooth manifold of dimension $n$, and let $\{(U_\alpha,\gamma_\alpha)\}_{\alpha\in\mathcal{A}}$, $\gamma_\alpha:M \supset U_\alpha \to V_\alpha\subset \R^{n}$, be an atlas for $M$. Then one defines
\begin{align}
\nonumber S_\delta^{m}(M):=&\Big\{a:T^*M\times (0,1]\to \C,\; a(\cdot, h)\in \Cinft(T^*M)\;\forall\;h\in(0,1],\\
&\quad (\gamma_\alpha^{-1}) ^*(\varphi_\alpha a)\in S_\delta^m(\R^{n})\;\,\forall\;\alpha\in\mathcal{A},\; \forall\;\varphi_\alpha\in \CT(U_\alpha)\Big\},\label{def:symbclass2}
\end{align}
where $(\gamma_\alpha^{-1})^\ast$ denotes the pullback along $\gamma^{-1}_\alpha\times \text{id}$.\footnote{The pullback is defined as follows: 
First, one identifies $T^*V_\alpha$ with $V_\alpha\times\R^n$. Then, given $a:T^*M\times (0,1]\to \C$, the function $(\varphi_\alpha a)\circ \big(\gamma^{-1}_\alpha\times (\partial{\gamma_\alpha^{-1}})^T\times \text{id}\big):V_\alpha\times\R^n\times(0,1]\to \C$ has compact support inside $V_\alpha$ in the first variable, and hence extends by zero to a function  
  $\R^{2n}\times(0,1]\to \C$. This function is defined to be $(\gamma_\alpha^{-1}) ^*(\varphi_\alpha a)$.} 
The definition is independent of the choice of atlas, and we call an element of 
$S_\delta^{m}(M)$ a \textit{symbol function}. We use the short hand notations $S_\delta^{-\infty}(M):=\cap_{m\in \R}S_\delta^{m}(M)$ and $S^m(M):=S_0^m(M)$, where $m\in \R\cup\{-\infty\}$. For such $m$ and $\delta\in [0,\frac{1}{2})$, we call a $\C$-linear map $P: \CT(M)\to \Cinft(M)$ \emph{semiclassical pseudodifferential operator on $M$ of order $(m,\delta)$} if the following holds:
\begin{enumerate}
\item For some (and hence any) atlas $\{(U_\alpha,\gamma_\alpha)\}_{\alpha\in\mathcal{A}}$,
$\gamma_\alpha:M \supset U_\alpha \to V_\alpha\subset \R^{n}$ of $M$ there exists a collection of symbol functions $\{s_\alpha\}_{\alpha\in\mathcal{A}}\subset S_\delta^m(\R^{n})$ such that for any two functions $\varphi_{\alpha,1}, \varphi_{\alpha,2} \in \CT(U_\alpha)$, it holds
\[
\varphi_{\alpha,1} P(\varphi_{\alpha,2} f)=\varphi_{\alpha,1} \mathrm{Op}_h(s_\alpha)((\varphi_{\alpha,2} f)\circ\gamma_\alpha^{-1})\circ\gamma_\alpha.
\]
\item For all $\varphi_{1}, \varphi_{2} \in \CT(M)$ with $\supp \varphi_{1}\cap \supp \varphi_{2}=\emptyset$, one has
\[
\norm{\Phi_1\circ P \circ \Phi_2}_{\Sob^{-N}(M)\to \Sob^N(M)}=\mathrm{O}(h^\infty)\quad\forall \;N=0,1,2,\ldots,
\]
where $\Phi_j$ is given by pointwise multiplication with $\varphi_j$ and $\Sob^N(M)$ is the $N$-th Sobolev space.
\end{enumerate}
When $\delta=0$, we just say \emph{order} $m$ instead of \emph{order} $(m,0)$. We denote by $\Psi^m_{h,\delta} (M)$ the $\C$-linear space of all semiclassical pseudodifferential operators on $M$ of order $(m,\delta)$, and we write $$\Psi_h^{m}(M):=\Psi^m_{h,0}(M),\qquad \Psi_h^{-\infty}(M)= \bigcap_{m\in \Z} \Psi_h^{m}(M).$$ 
From the classical theorems about pseudodifferential operators one infers in particular the following relation between symbol functions and semiclassical pseudodifferential operators, see \cite[page 86]{hoermanderIII}, \cite[Theorem 14.1]{zworski}, \cite[page 383]{dyatlov}. There is a $\C$-linear map 
\begin{equation}
\Psi^m_{h,\delta} (M) \to S_\delta^{m}(M)/\left(h^{1-2\delta}S_\delta^{m-1}(M)\right),\quad P  \mapsto \sigma(P) \label{eq:sigma},
\end{equation}
which assigns to a semiclassical pseudodifferential operator its \emph{principal symbol}. Moreover, for each choice of atlas  $\{(U_\alpha,\gamma_\alpha)\}_{\alpha\in\mathcal{A}}$ of $M$ and a partition of unity $\{\varphi_\alpha\}_{\alpha \in \mathcal{A}}$ subordinate to $\{U_\alpha\}_{\alpha \in \mathcal{A}}$, there is a $\C$-linear map called \emph{quantization}, written
\begin{equation}
S_\delta^{m}(M) \to \Psi^m_{h,\delta} (M), \quad s \mapsto \mathrm{Op}_{h,\left\{U_\alpha,\varphi_\alpha\right\}_{\alpha \in \mathcal{A}}}(s).\label{eq:op}
\end{equation}
Any choice of such a map induces the same $\C$-linear bijection
\begin{align}
\Psi^m_{h,\delta} (M)/\left(h^{1-2\delta}\Psi^{m-1}_{h,\delta}(M)\right) & \begin{matrix}\sigma\\ \rightleftarrows \\ \mathrm{Op}_h \end{matrix} S_\delta^{m}(M)/\left(h^{1-2\delta}S_\delta^{m-1}(M)\right),\label{eq:qmaps}
\end{align}
which means in particular that the bijection exists and is independent from the choice of atlas and partition of unity. We call an element in a quotient space $S_\delta^{m}(M)/\left(h^{1-2\delta}S_\delta^{m-1}(M)\right)$ a \textit{principal symbol}, whereas we call the elements of $S_\delta^{m}(M)$ \textit{symbol functions}, as introduced above. For a semiclassical pseudodifferential operator $A$, we use the notation 
\[
\sigma(A)=[a]
\]
to express that the principal symbol $\sigma(A)$ is the equivalence class in $S_\delta^{m}(M)/\left(h^{1-2\delta}S_\delta^{m-1}(M)\right)$ defined by a symbol function $a\in S_\delta^{m}(M)$. 
\subsection{Semiclassical functional calculus for Schr\"odinger operators on compact manifolds}
\label{sec:schroedinger}
In what follows, let  $M$ be a connected compact Riemannian manifold without boundary, and $P(h)$ a Schr\"odinger operator on $M$ as in   \eqref{eq:13.08.15}. 
There exists a well-known functional calculus \cite[Theorems 14.9 and 14.10]{zworski} for such operators by which  for each $f\in \mathcal{S}(\mathbb{R})$ the operator $f(P(h))$ defined by the spectral theorem
for unbounded self-adjoint operators is an element in $\Psi_h^{-\infty}(M)$. As can be shown, $f(P(h))$ extends to a bounded operator $f(P(h)):\L^2(M)\to \L^2(M)$ of trace class, and the principal symbol of $f(P(h))$ is represented by $f\circ p$, where  $p$ is given by \eqref{eq:13.08.15a}. Moreover, 
\bqn 
\tr  f(P(h))= \frac{1}{(2\pi h)^n}\intop_{T^*M} (f\circ p)\d(T^*M)\;+\;\mathrm{O}\big(h^{-n+1}).
\eqn
Nevertheless, for our purposes, this is not enough, and we  need a functional calculus for Schr\"odinger operators where the test function $f$ is allowed to depend on $h$. Such a calculus has been recently developed by  K\"uster \cite{kuester15} giving explicit formulas for the global operator $f(P(h))$ in terms of  locally defined quantizations in $\R^n$  with trace norm remainder estimates. The class of $h$-dependent test functions that we will consider is given by $\bigcup_{\delta\in[0,\frac{1}{2})}\mathcal{S}^{\mathrm{bcomp}}_{\delta}$, where for each $\delta\in[0,\frac{1}{2})$ the symbol $\mathcal{S}^{\mathrm{bcomp}}_{\delta}$ denotes the set of all families $\{f_h\}_{h\in (0,1]}\subset \CT(\R)$ such that 
\begin{itemize}
\item[(1)] the $j$-th derivative $f_h^{(j)}$ of $f_h$ fulfills
\[
\norm{f_h^{(j)}}_\infty=\mathrm{O}(h^{-\delta j})\qquad\text{as }h\to 0,\quad j=0,1,2,\ldots;
\]
\item[(2)] there is a compact Interval $I\subset \R$ such that the support of $f_h$ is contained\footnote{This condition can be relaxed to allow also supports with growing diameter and varying position as $h\to 0$, see \cite{kuester15}.} in $I$ for all $h$.
\end{itemize}
One then has the following
\begin{thm}[{\cite[Theorems 4.4, 4.5 and Corollary 4.6]{kuester15}}]\label{thm:semiclassfunccalcmfld}Let $\delta\in [0,\frac{1}{2})$ and $\rho_h\in \mathcal{S}^{\mathrm{bcomp}}_{\delta}$. Then, for small $h$ the operator $\rho_h(P(h))$ is a semiclassical pseudodifferential operator on $M$ of order $(-\infty,\delta)$, and its principal symbol is represented by $\rho_h\circ p$. In addition, consider an operator $B\in \Psi_{h,\delta}^0(M)\subset \B(\L^2(M))$ with principal symbol $[b]$, where $b\in S_\delta^0(M)$. Then the following assertions hold:
\begin{itemize}
\item
For any finite atlas $\{U_\alpha,\gamma_\alpha\}_{\alpha \in \mathcal{A}}$, $\gamma_\alpha: U_\alpha\stackrel{\simeq}{\to} \R^n$, with a subordinated compactly supported partition of unity $\{\varphi_\alpha\}_{\alpha \in \mathcal{A}}$ on $M$  and a family of functions $\{\overline \varphi_\alpha,\overline{\overline \varphi}_\alpha,\overline{\overline {\overline\varphi}}_\alpha\}_{\alpha \in \mathcal{A}}\subset \CT(M)$ such that $\supp \overline \varphi_\alpha,\supp \overline{\overline \varphi}_\alpha, \supp \overline{\overline{\overline\varphi}}_\alpha\subset U_\alpha$ and $\overline \varphi_\alpha\equiv 1$ on $\supp  \varphi_\alpha$, $\overline{\overline \varphi}_\alpha\equiv1$ on $\supp  \overline \varphi_\alpha$, and $\overline{\overline{\overline\varphi}}_\alpha\equiv 1$ on $\supp  \overline{\overline \varphi}_\alpha$, there is a number $h_0\in(0,1]$ such that for each $N\in \N$ there is a collection of symbol functions $\{r_{\alpha,\beta,N}\}_{\alpha,\beta\in \mathcal{A}}\subset h^{1-2\delta}S^0_\delta(\R^n)$ and an operator $\mathfrak{R}_{N}(h)\in \B(\L^2(M))$ such that one has\begin{multline}
B\circ \rho_h(P(h))(f)=\sum_{\alpha \in \mathcal{A}}\overline \varphi_\alpha\cdot\mathrm{Op}_h(u_{\alpha,0})\big((f\cdot \overline{\overline {\overline \varphi}}_\alpha)\circ \gamma_\alpha^{-1}\big)\circ\gamma_\alpha\\
+\sum_{\alpha,\beta \in \mathcal{A}}\overline \varphi_\beta\cdot\mathrm{Op}_h(r_{\alpha,\beta,N})\big((f\cdot \overline{\overline \varphi}_\alpha\cdot\overline {\overline {\overline \varphi}}_\beta)\circ\gamma_\beta^{-1}\big)\circ\gamma_\beta  + \mathfrak{R}_{N}(h)(f)\qquad\forall\; f\in \L^2(M),\;h\in (0,h_0],\label{eq:newassertion33}
\end{multline}
where\footnote{Here $\tau:T^*M\to M$ denotes the bundle projection, and the composition with $(\gamma_\alpha^{-1},(\partial\gamma_\alpha^{-1})^T)$ means that we compose with the inverse chart $\gamma_\alpha^{-1}$ in the manifold variable and with the adjoint of its derivative in the cotangent space variable.}
\bq
u_{\alpha,0}=\big((\rho_h\circ p)\cdot b\cdot (\varphi_\alpha\circ \tau)\big)\circ(\gamma_\alpha^{-1},(\partial\gamma_\alpha^{-1})^T).\label{eq:princsymbsymb}
\eq
\item The operator $\mathfrak{R}_N(h)\in \B(\L^2(M))$ is of trace class and its trace norm fulfills
\bq
\norm{\mathfrak{R}_N(h)}_{\text{tr},\L^2(M)}=\mathrm{O}\big(h^{N}\big)\quad\text{as }h\to 0.\label{eq:tracenormestim343434}
\eq
\item For fixed $h\in (0,h_0]$, each symbol function $r_{\alpha,\beta,N}$ is an element of $\CT(\R^{2n})$ that fulfills 
\bq
\supp  r_{\alpha,\beta,N}\subset \supp   \big((\rho_h\circ p)\cdot b \cdot(\varphi_\alpha\circ \tau)\big)\circ(\gamma_\alpha^{-1},(\partial\gamma_\alpha^{-1})^T).\label{eq:wts798978}
\eq
\item The trace of $B\circ\rho_h(P(h))$ is given by the asymptotic formula
\begin{multline}
\tr \,\big[B\circ\rho_h(P(h))\big]=\tr \,\big[\rho_h(P(h))\circ B\big]\\
=\frac{1}{(2\pi h)^n}\int_{T^*M}b\cdot(\rho_h\circ p)\d(T^*M)\;+\;\mathrm{O}\Big(h^{1-2\delta-n}\text{vol}_{\,T^*M}\big[\supp \big(b\cdot (\rho_h\circ p)\big)\big]\Big)\qquad\text{as }h\to 0.\label{eq:result474747433}
\end{multline}
\end{itemize}
\end{thm}
\subsection{Actions of compact Lie groups and symplectic reduction}
\label{sec:2.3}
In what follows, we recall some essential facts from
the general theory of compact Lie group actions on smooth manifolds. For a detailed introduction, we refer the reader to  \cite[Chapters I, IV, VI]{bredon}. Let $\Xbf$ be a
smooth manifold of dimension $n$ and $G$ a  Lie group acting
locally smoothly on $\Xbf$.
For $x\in \Xbf$,  denote by $G_{x}$ the isotropy group and by $G\cdot x=\O_x$
the $G$-orbit through $x$ so that
\begin{equation*}
G_{x} =  \{g\in G,\; g\cdot x=x\},\qquad
\O_x =G\cdot x  =  \{g\cdot x\in \Xbf,\; g\in G\}.
\end{equation*}
Note that $G\cdot x$ and $G/G_{x}$ are homeomorphic.
The equivalence class of an orbit $\O_x$ under equivariant homeomorphisms, written $[\O_x]$, 
is called its \textit{orbit type}. The conjugacy class 
of a stabilizer group $G_{x}$ is called its \textit{isotropy type}, and written $(G_{x})$. 
If $K_{1}$ and $K_{2}$ are closed subgroups of $G,$ a partial ordering
of orbit and isotropy types is given by
\[
[G/K_{1}]\leq[G/K_{2}]\Longleftrightarrow(K_{2})\leq(K_{1})\Longleftrightarrow K_{2}\textrm{ is conjugate to a subgroup of }K_{1}.
\]
For any closed subgroup $K\subset G$, one denotes by $\Xbf(K):=\mklm{x \in  \Xbf: G_x \sim K} $ the subset of points of isotropy type $\left(K\right)$.  

Assume now that $G$ is compact. The set of all orbits is denoted by $\Xbf/G$, and equipped with the quotient topology it becomes a topological  Hausdorff space \cite[Theorem I.3.1]{bredon}. In the following we shall assume that it is connected. 
One of the central results in the theory of compact group actions is the \emph{principal orbit theorem} \cite[Theorem IV.3.1]{bredon}, which states that there exists a maximum orbit type $\left[\O_{\textrm{max}}\right]$
with associated minimal isotropy type $\left(H\right)$. Furthermore, $\Xbf(H)$  is open and dense
in $\Xbf$, and its image in $\Xbf/G$ is connected. We call $\left[\O_{\textrm{max}}\right]$ the \textit{principal orbit
type} of the $G$-action on $\Xbf$ and a representing orbit a \textit{principal orbit}. Similarly, we call the isotropy type $\left(H\right)$
and an isotropy group $G_{x}\sim H$ \textit{principal}. Casually,
we will identify orbit types with isotropy types and say \emph{an orbit of type $\left(H\right)$} or even \emph{an orbit of type $H$}, making no distinction between equivalence classes and their representants. The reduced space $\Xbf(H)/G$ is a smooth
manifold of dimension $n-\kappa$, where $\kappa$ is the dimension
of $\O_{\textrm{max}}$, since $G$ acts with only one orbit type on $\Xbf(H).$

Next, let us briefly recall some central results from the  theory of symplectic reduction of Marsden and Weinstein, Sjamaar, Lerman and Bates.
For a detailed exposition of these facts we refer the reader to \cite{ortega-ratiu}. Thus, let $(\Xbf,\omega)$ be a connected symplectic manifold with an action of an arbitrary  Lie
group $G$ that leaves $\omega$ invariant. In particular, we will be interested in the case where $\Xbf=T^\ast M$ is the co-tangent bundle of a smooth manifold $M$. Denote by $\mathfrak{g}$
the Lie algebra of $G$. Note that $G$ acts on $\mathfrak{g}$ via
the adjoint action and on $\mathfrak{g}^{*}$ via the co-adjoint
action. The group $G$ is said to act on $\Xbf$ in a \emph{Hamiltonian fashion}, if for each $X \in \g$  there exists a  $\Cinft$-function $\Jbb_{X}:\Xbf  \to  \mathbb{R}$ depending linearly on $X$ such that the fundamental vector field $\widetilde X$ on $\Xbf$ associated to $X$  is given by the Hamiltonian vector field of $\Jbb_X$. One then has
\[
d\Jbb_{X}=-\widetilde{X}\, \lrcorner\, \omega, 
\]
and one defines the \emph{momentum map} of the Hamiltonian action as the equivariant map
\begin{eqnarray*}
\Jbb:\Xbf  \to  \g^\ast,\qquad \Jbb(\eta)(X)=  \Jbb_X(\eta),
\end{eqnarray*}
see \cite[Section 4.5]{ortega-ratiu}.
It is clear from the definition that a momentum map is unique up to
addition of a constant function, and it satisfies $\Ad^\ast(g^{-1}) \circ \Jbb= \Jbb \circ g$.  
In this case $(\Xbf,\omega, \Jbb)$ is called a \emph{Hamiltonian $G$-space}, and one  defines 
\bqn 
\Omega:=\Jbb^{-1}(\mklm{0}), \qquad \widetilde \Omega:=\Omega/G.
\eqn
Unless the $G$-action on $\Xbf$ is free, the \emph{reduced space} or \emph{symplectic quotient} $\widetilde \Omega$ will in general not be a smooth manifold, but a topological quotient space. One can show that  $\widetilde \Omega$  possesses a stratification into smooth symplectic manifolds \cite[Theorem 8.1.1]{ortega-ratiu}. 
Let now   $M$ be a connected compact boundary-less  Riemannian manifold of dimension $n$, carrying an isometric  effective action of a compact connected 
Lie group $G$. Then  $\Xbf=T^\ast M$ constitutes a Hamiltonian $G$-space when endowed with the canonical symplectic structure and the $G$-action induced from the smooth action on $M$, and one has
\begin{equation}
\Omega=\Jbb^{-1}(\{0\})=\bigsqcup_{x\in M}\textrm{Ann}\,T_x(G\cdot x)\label{eq:omega},
\end{equation}
where $\textrm{Ann}\,V_x\subset T^\ast_x M$ denotes the annihilator of a  subspace $V_x \subset T_xM$. Clearly, as soon as there are two orbits $G\cdot x$, $G\cdot x'$ in $M$ of different dimensions, their annihilators $\textrm{Ann}\,T_x(G\cdot x)$ and $\textrm{Ann}\,T_x(G\cdot x')$ have different dimensions, so that $\Omega$ is not a  vector bundle in that case. Further, let 
\begin{align*}
M_{\textrm{reg}}&:=M(H), \qquad 
\Omega_{\textrm{reg}}:=\Omega \cap (T^{*}M)(H),
\end{align*}
where $M(H)$ and $(T^{*}M)(H)$ denote the union of orbits of type $(H)$ in $M$ and $T^{*}M$, respectively. By the principal orbit theorem  $M_{\textrm{reg}}$ is open (and dense)  in $M$, and therefore  a smooth submanifold.  We then define
 $$\widetilde{M}_{\textrm{reg}}:=M_{\textrm{reg}}/G.$$
$\widetilde{M}_{\textrm{reg}}$   is a smooth boundary-less manifold, since $G$ acts on ${M}_{\textrm{reg}}$ with only one orbit type and $M_{\textrm{reg}}$ is open in $M$.  Moreover, because the Riemannian metric on $M$ is $G$-invariant, it induces a Riemannian metric on $\widetilde{M}_{\textrm{reg}}$. On the other hand, by \cite[Theorem 8.1.1]{ortega-ratiu}, $\Omega_{\textrm{reg}}$ is a smooth submanifold of $T^{*}M$, and  the quotient
$$
\widetilde{\Omega}_{\textrm{reg}}:= \Omega_{\textrm{reg}}/G
$$
possesses a unique differentiable structure such that the projection $\pi: \Omega_{\textrm{reg}} \rightarrow \widetilde{\Omega}_{\textrm{reg}}$ is a surjective submersion. Furthermore, there exists a unique symplectic form $\widetilde{\omega}$ on $\widetilde{\Omega}_{\textrm{reg}}$ such that $
\iota^\ast \omega = \pi^\ast \widetilde{\omega}$,
where $\iota:  \Omega_{\textrm{reg}} \hookrightarrow T^{*}M$ denotes the inclusion and $\omega$  the canonical symplectic form on $T^{*}M$. Consider now  the inclusion $j:(T^*M_\text{reg}\cap \Omega)/G\hookrightarrow \widetilde{\Omega}_\text{reg}$. The symplectic form $\widetilde{\omega}$ on $\widetilde{\Omega}_\text{reg}$ induces a symplectic form $j^*\widetilde{\omega}$ on $(T^*M_\text{reg}\cap \Omega)/G$.  We then have the following

\begin{lem}[\bf Singular co-tangent bundle reduction]\label{lem:isomorphic}
Let $\widehat{\omega}$ denote the canonical symplectic form on the co-tangent  bundle $T^*\widetilde{M}_\text{reg}$. Then the two $2(n-\kappa)$-dimensional symplectic manifolds 
\bqn
\left((T^*M_\text{reg}\cap \Omega)/G,j^*\widetilde{\omega}\right)\, \simeq \,  (T^*\widetilde{M}_\text{reg},\widehat{\omega})
\eqn
 are canonically symplectomorphic. 
\end{lem}
\begin{rem}
\label{rem:2.6}
If $M=M_{\text{reg}}$, the previous lemma simply asserts that  
\bq
\label{eq:0304}
 T^\ast \widetilde M\simeq \widetilde \Omega=\Omega /G
 \eq  
  as symplectic manifolds. In case that $G$ acts on $M$ only with finite isotropy groups, $\widetilde M$ is an orbifold, and  the relation \eqref{eq:0304} holds as well, being the quotient presentation of the co-tangent bundle of $\widetilde M$ as an orbifold \cite{kordyukov12}. 
\end{rem}

\begin{proof}First, we apply \cite[Theorem 8.1.1]{ortega-ratiu} once to the manifold $T^*M$ and once to the manifold $T^*M_\text{reg}$. Noting that the momentum map of the $G$-action on $T^*M_\text{reg}$ agrees with the restriction of the momentum map of the $G$-action on $T^*M$ to $T^*M_\text{reg}$, we get that $j^*\widetilde{\omega}$ is the unique symplectic form on $(T^*M_\text{reg}\cap \Omega_\text{reg})/G$ which fulfills
\begin{equation}
i^*\omega=\Pi^*j^*\widetilde{\omega},\label{eq:uniqueform}
\end{equation}
where $\Pi:T^*M_\text{reg}\to T^*M_\text{reg}/G$ is the orbit projection, $i:T^*M_\text{reg}\cap\Omega_\text{reg}\hookrightarrow T^*M_\text{reg}$ is the inclusion, and $\omega$ is the canonical symplectic form on $T^*M_\text{reg}$.
The rest of the proof is now essentially the proof of the standard {co-tangent  bundle reduction theorem} \cite[Theorem 2.2.2]{marsden} for the manifold $M_\text{reg}$. A detailed proof of the present lemma is also given in \cite{kuester}.
\end{proof}

\subsection{Measure spaces and group actions}\label{subsec:spaces}
In what follows, we  give an overview of the spaces and measures that will be relevant in the upcoming sections. If not stated otherwise, measures are not assumed to be normalized.  As before, let  $M$ be a connected compact boundary-less  Riemannian manifold of dimension $n$ with  Riemannian volume density $dM$, carrying an isometric  effective action of a compact connected Lie group $G$ with Haar measure $dg$. Let $\kappa$ be the dimension of an orbit of principal type. Note that if $\text{dim }G>0$, $dg$ is equivalent to the normalized Riemannian volume density  on $G$ associated to any choice of left-invariant Riemannian metric on $G$. If $\text{dim }G=0$, which in our case implies that $G$ is trivial, $dg$ is the normalized counting measure. Consider further $T^\ast M$ with its canonical symplectic form $\omega$, endowed with the natural Sasaki metric. Then the  Riemannian volume density $d(T^\ast M)$ given by the Sasaki metric coincides with the symplectic  volume form $\omega^n/n!$, see \cite[page 537]{katok}. Next, if $\Omega:=\Jbb^{-1}(\{0\})$ denotes the zero level of the momentum map $\Jbb:T^\ast M \rightarrow \g^\ast$ of the underlying Hamiltonian action, we regard $\Omega_{\textrm{reg}}\subset T^*M$ as a Riemannian submanifold with Riemannian metric  induced by the Sasaki metric on $T^*M$, and denote the associated  Riemannian volume density by $\d\Omega_{\textrm{reg}}$.
Similarly, let
 \bq
 \label{eq:Ccal}
\mathcal{C}:=\{(\eta,g)\in\Omega\times G:\; g\cdot\eta=\eta\}.
\eq
As $\Omega$, the space $\mathcal{C}$ is not a manifold in general. We consider therefore the space $\textrm{Reg}\;\mathcal{C}$
of all regular points in $\mathcal{C}$, that is,  all points that have
a neighbourhood which is a smooth manifold. $\textrm{Reg}\;\mathcal{C}$
is a smooth, non-compact submanifold of $T^{*}M\times G$ of co-dimension $2\kappa$, and it is not difficult to see that 
\[
\textrm{Reg}\;\mathcal{C}=\{(\eta,g)\in\Omega\times G,\; g\cdot\eta=\eta,\;\eta\in \Omega_{\textrm{reg}}\},
\]
see e.g.\ \cite[(17)]{ramacher10}.  We then regard  $\text{Reg }\mathcal{C}\subset T^*M\times G$ as a Riemannian submanifold with Riemannian metric  induced by the product metric of the Sasaki metric on $T^*M$ and some left-invariant Riemannian metric on $G$, and denote the corresponding Riemannian volume density by $d(\text{Reg }\mathcal{C})$. In the same way, if $x\in M$ and $\eta \in T^*M$ are points,  the orbits $G\cdot x$ and $G\cdot\eta$ are smooth submanifolds of $M$ and $T^*M$, respectively, and if they have dimension greater than zero, we endow them with the corresponding  Riemannian orbit measures, denoted by  $d\mu_{G\cdot x}$ and $d\mu_{G\cdot \eta}$, respectively. If the dimension of an orbit is zero, it is a finite collection of isolated points, since $G$ is compact, and we define $d\mu_{G\cdot x}$ and $d\mu_{G\cdot \eta}$ to be the counting measures. Further, for any Riemannian $G$-space ${\bf X}$, we  define the orbit volume functions 
\bqn 
\vol_\O:{\bf X} \rightarrow (0,\infty), \quad x \mapsto \vol \O_x=\vol (G\cdot x), \qquad \vol:{\bf X}/G \rightarrow (0,\infty), \quad \O \mapsto \vol \O.
\eqn
Note that by definition we have  $\text{vol}>0$, $\vol_\O>0$ for all orbits, singular or not, since the orbit volume is defined using the induced Riemannian measure for smooth orbits and the counting measure for finite orbits.  An important property of the orbit measures is their relation to the normalized Haar measure on $G$. Namely, for any orbit $G\cdot x$ and any continuous function $f:G\cdot x\to \C$, we have
\begin{equation}
\int_{G\cdot x}f(x')\d\mu_{G\cdot x}(x')=\vol(G\cdot x)\int_{G}f(g\cdot x )\d g.\label{eq:orbithaar}
\end{equation}
To see why (\ref{eq:orbithaar}) holds, recall that there is a $G$-equivariant diffeomorphism $\Phi:G\cdot x\to G/G_x$. Then $\Phi_\ast  (d\mu_{G\cdot x})$ is a $G$-invariant finite measure on $G/G_x$. Similarly, if $\Pi: G\to G/G_x$ denotes the canonical projection,  $\Pi_\ast (dg)$ is also a $G$-invariant finite measure on $G/G_x$. Hence, $\Phi_\ast (d\mu_{G\cdot x})=C\cdot \Pi_\ast (dg)$ for some constant $C$ which is precisely given by $\text{vol}(G\cdot x)$, since $\Pi_\ast( dg)$ is normalized. Observing that $\intop_{G}f(gx)\d g=\intop_{G/G_x}f(gG_xx)\Pi_\ast( dg)$,  (\ref{eq:orbithaar}) follows.

We describe now the quotient spaces and  measures on them that will be relevant to us. Let  $d\widetilde{M}_{\textrm{reg}}$ be the Riemannian volume density on $\widetilde{M}_{\textrm{reg}}$ associated to the Riemannian metric on $\widetilde{M}_{\textrm{reg}}$ induced by the $G$-invariant metric on $M$.  Regarding the co-tangent  bundle $T^*\widetilde{M}_{\textrm{reg}}$, we endow it with the canonical symplectic structure and let  $d(T^*\widetilde{M}_{\textrm{reg}})$ be the corresponding  symplectic volume form. Again, it coincides with the Riemannian volume form given by the natural Sasaki metric on $T^*\widetilde{M}_{\textrm{reg}}$. Similarly, the symplectic stratum $\widetilde{\Omega}_{\textrm{reg}}$ carries a canonical symplectic form $\widetilde \omega$ from \cite[Theorem 8.1.1]{ortega-ratiu}, and   $d\widetilde{\Omega}_{\textrm{reg}}=\widetilde{\omega}^{(n-\kappa)}/(n-\kappa)!$ denotes the corresponding symplectic volume form. One can then show that  $d\widetilde{\Omega}_{\textrm{reg}}$ agrees with the Riemannian volume density associated to the Riemannian metric on $\widetilde{\Omega}_{\textrm{reg}}$ induced by the Riemannian metric on $\Omega_{\textrm{reg}}$, see Lemma \ref{lem:pushfwd}. Since orbit projections on principal strata define fiber bundles \cite[Theorem IV.3.3]{bredon}, Lemma \ref{lem:pushfwd} implies that $d\mu_{G\cdot x}$  and $d\mu_{G\cdot\eta}$  are the unique measures on the orbits in $M_\textrm{reg}$ and  $\Omega_\textrm{reg}$  such that
\begin{align}
\int_{M_\text{reg}}f(x)\,dM(x)&=\int_{\widetilde{M}_\text{reg}}\quad\int_{G\cdot x}f(x')\,d\mu_{G\cdot x}(x')\,d\widetilde{M}_{\textrm{reg}}(G\cdot x)\quad\forall\,f\in C(M_\text{reg}), \label{eq:prop80} \\
\int_{\Omega_\textrm{reg}}f(\eta)\,d(\Omega_\textrm{reg})(\eta) &=\int_{\widetilde{\Omega}_{\textrm{reg}}}\quad\int_{G\cdot\eta}f(\eta')\,d\mu_{G\cdot\eta}(\eta')\,d\widetilde{\Omega}_{\textrm{reg}}(G\cdot\eta)\quad\forall\,f\in C(\Omega_\text{reg}).\label{eq:prop81}
\end{align}

Next, hypersurfaces will be endowed with the  measures induced by the measures on the ambient manifold, see Lemma \ref{lem:hypersurf0}. Thus, for our Hamiltonian function $p:T^\ast M\rightarrow \R $ with regular value $c\in \R$, there is a canonical  measure $d\Sigma_c$ on the hypersurface $\Sigma_c={p}^{-1}(\{c\})$, induced by the symplectic volume form on $T^*M$, {or equivalently, by the Riemannian volume density associated to the Sasaki metric}. In the case $\Sigma_c=S^\ast M$ it is called the \emph{Liouville measure} and denoted by $d(S^*M)$. 
Similarly, for $S^*\widetilde{M}_{\textrm{reg}}$, the unit co-sphere bundle over $\widetilde{M}_{\textrm{reg}}$, the induced measure $d(S^*\widetilde{M}_{\textrm{reg}})$ is the Liouville measure, and for  a general hypersurface $\widetilde{\Sigma}_c:=\widetilde{p}^{-1}(\{c\})$, where  $\widetilde{p}$ is induced by $p$ and $\widetilde{\Omega}_{\textrm{reg}}$ is endowed with the measure $d\widetilde{\Omega}_\text{reg}$,  we denote the induced hypersurface measure by  $d\widetilde{\Sigma}_c$. Furthermore, since the intersection is transversal, $\Sigma_c\cap \Omega_\text{reg}$ is a smooth hypersurface of $\Omega_\text{reg}$, and therefore carries a  measure $\mu_c$   induced by $d\Omega_\text{reg}$. As the orbit projection $\Sigma_c\cap \, \Omega_{\text{reg}}\to\widetilde{\Sigma}_{c}$ is a fiber bundle, $\mu_c$ fulfills
\bq
\label{eq:2.15}
\intop_{\Sigma_c\cap \, \Omega_{\text{reg}}}f(\eta)\d \mu_c(\eta)=\intop_{\widetilde{\Sigma}_c}\quad\intop_{G\cdot \eta}f(\eta')\,d\mu_{G\cdot \eta}(\eta')\,d\widetilde{\Sigma}_{c}(G\cdot \eta)\quad\forall\,f\in C(\Sigma_c \cap \, \Omega_{\text{reg}}).
\eq
Finally, let $d{\widetilde M}:=\pi_{\ast}dM$ be the pushforward of $dM$ along the canonical projection $\pi:M\to \widetilde M:=M/G$. This means that, for $f\in C(\widetilde M)$, we have 
$$\int_{\widetilde M}f(\O)\d{\widetilde M}(\O)= \int_{M}f\circ\pi(x)\d M(x).$$

In what follows, we will use the orbit volume functions $\text{vol}$ and $\vol_\O$ together with the previously defined measures to form new {measures}.  This way we obtain  on $\widetilde M$ the measure  $\frac{d{\widetilde M}}{\vol}$, and on $\Sigma_c\cap \, \Omega_{\text{reg}}$ the measure $\frac{\mu_c}{\vol_\O}$. These measures are of fundamental importance in this paper and its sequel \cite{kuester-ramacher15b}. Finally,  for a measure space $(\Xbf,\mu)$  with $\mu(\Xbf)<\infty$ and a measurable function $f$ on $\Xbf$,  we shall use the common notation  
$$\fint_\Xbf f \d \mu := \frac 1 {\mu(\Xbf)} \int_\Xbf f \d\mu.$$

\subsection{Singular equivariant asymptotics}  As it will become apparent in the next section, our results rely on the description of  the asymptotic behavior of certain oscillatory integrals that were already examined in \cite{ramacher10, ramacher15a} while studying  the spectrum of an invariant elliptic operator. Thus, let $M$ be a Riemannian manifold of dimension $n$ carrying a smooth effective action of a connected compact Lie group $G$. Consider a chart $\gamma:M\supset U\overset{\simeq}{\to}V\subset\mathbb{R}^{n}$
on $M$, and write $(x,\xi)$ for an element in $T^*U \simeq U \times \R^n$ with respect to the canonical trivialization of the co-tangent  bundle over the chart domain. Let $a_\mu\in \CT( T^{*}U\times G)$ be an amplitude that might depend on a parameter  $\mu\in\mathbb{R}_{>0}$  such that $(x,\xi,g) \in \supp a_\mu$ implies $g \cdot x \in U$, and assume that there is a compact $\mu$-independent set $\mathcal{K}\subset T^{*}U\times G$ such that $\supp a_\mu\subset \mathcal{K}$ for each $\mu$. Further, consider the phase  function 
\bq\Phi(x,\xi,g):=\left\langle \gamma(x)-\gamma(g\cdot x ),\xi\right\rangle, \qquad (x,\xi, g) \in \supp a_\mu,
\label{eq:phase} 
\eq
where $\left\langle \cdot,\cdot\right\rangle $ denotes the Euclidean scalar
product on $\mathbb{R}^{n}$. It represents a global analogue of the momentum map, and oscillatory integrals with phase function given by the latter have been examined in \cite{ramacher13} in the context of equivariant cohomology.  The phase function $\Phi$ has the critical set  \cite[Equation following (3.3)]{ramacher10}
 \begin{align*}
\Crit(\Phi)&=\mklm{(x,\xi,g) \in T^\ast U \times G : (\Phi_{\ast})_{(x,\xi,g)}=0} \\
&= \mklm{( x  ,\xi,g) \in (\Omega \cap T^\ast U)\times G :  \,  g \cdot (x,\xi)=(x,\xi)}=\mathcal{C}\cap T^\ast U,
\end{align*}
with $\Ccal$ given as in \eqref{eq:Ccal},  and the central question  is   to describe the asymptotic behavior as $\mu \to + \infty$ of oscillatory integrals of the form
\begin{equation}
I(\mu)=\int_{T^{*}U}\int_{G}e^{i\mu\Phi(x,\xi,g)}a_\mu(x,\xi,g)\, dg\, d\left(T^{*}U\right)(x,\xi).\label{eq:integral}
\end{equation}
The major difficulty here resides in the fact that, unless the $G$-action on $T^\ast M$ is free, the considered momentum map is not a submersion, so that the zero set $\Omega$ of the momentum map and the critical set of the phase function $\Phi$ are not smooth manifolds. The stationary phase theorem can therefore not immediately be applied to the integrals  $I(\mu)$. Nevertheless, it was shown in \cite{ramacher10}   that by constructing a strong resolution of the set 
\bqn 
\Ncal=\mklm{(p,g) \in M \times G : g\cdot p = p}
\eqn
a partial desingularization $\mathcal{Z}: \widetilde {\bf X} \rightarrow {\bf X}:= T^\ast M \times G $ of the critical  set $\mathcal{C}$ can be achieved, and applying the stationary phase theorem in the resolution space $\widetilde {\bf X}$, an asymptotic description of $I(\mu)$ can be obtained. Indeed, the map $\mathcal{Z}$ yields 
a partial monomialization of  the local ideal $I_\Phi=(\Phi)$ generated by the phase function \eqref{eq:phase} according to  
\bqn 
\mathcal{Z}^\ast (I_\Phi) \cdot \E_{\tilde x, \widetilde {\bf X}} = \prod_j\sigma_j^{l_j}   \cdot\mathcal{Z}^{-1}_\ast(I_\Phi) \cdot \E_{\tilde x, \widetilde {\bf X}},
\eqn
 where $\E_{\widetilde{\bf X}}$ denotes the structure sheaf of rings of $\widetilde{\bf X}$, $\mathcal{Z}^\ast (I_\Phi)$ the total transform, and $\mathcal{Z}^{-1}_\ast(I_\Phi)$ the weak transform of $I_\Phi$, while the  $\sigma_j$ are local coordinate functions near each $\tilde x \in \widetilde {\bf X}$ and the $l_j$ natural numbers. As a consequence, the phase function factorizes locally according to $\Phi \circ \mathcal{Z} \equiv \prod \sigma_j^{l_j} \cdot  \tilde \Phi^ {wk}$,
and one shows  that  the weak transforms $ \tilde \Phi^ {wk}$ have clean critical sets. 
Asymptotics for the integrals $I(\mu)$ are  then obtained by pulling  them back to the resolution space $\widetilde {\bf X}$, and applying  the stationary phase theorem to the $\tilde \Phi^{wk}$  with the variables  $\sigma_j$ as parameters.  Thus, with  $\kappa$ and $\Lambda$ as in Section \ref{sec:13.08.15} one has
\begin{thm}[{\cite[Theorem 2.1]{ramacher15a}}]
\label{thm:main}
In the limit $\mu\to+\infty$ one has
\begin{align}
\begin{split}
\Big | I(\mu)&-\left(\frac{2\pi}{\mu}\right)^{\kappa}\intop_{\textrm{Reg} \, {\mathcal{C}}}\frac{a_\mu(x,\xi,g)}{\left|\det\Phi''(x,\xi,g)_{|N_{(x,\xi,g)}Reg\, {\mathcal{C}}}\right|^{1/2}}\, d(\textrm{Reg}\,\mathcal{\mathcal{C}})(x,\xi,g)\Big | \\
&\leq C \sup_{l \leq 2\kappa+3} \norm{D^l a_\mu}_{\infty}  \mu^{-\kappa-1}(\log \mu)^{{\Lambda^{}}-1},\label{eq:I(mu)}
\end{split}
\end{align}
where  $D^l$ is a differential operator of order $l$ independent of $\mu$ and $a_\mu$ and $C>0$ a constant independent of $\mu$ and $a_\mu$, too. The expression $\Phi''(x,\xi,g)_{N_{(x,\xi,g)}Reg\,\mathcal{\mathcal{C}}}$
denotes the restriction of the Hessian of $\Phi$ to the normal space
of $Reg\,\mathcal{\mathcal{C}}$ inside $T^{*}U\times G$
at the point $(x,\xi,g)$. In particular, the integral in  (\ref{eq:I(mu)})
exists.
\end{thm}

The precise form of the remainder estimate in the previous theorem will allow us  to give remainder estimates also in  the case when the amplitude depends on $\mu$. To conclude, let us note the following 
\begin{lem}
\label{thm:integralconcrete}
Let $b\in C_{c}^{\infty}(\Omega\cap T^\ast U)$ and $\chi\in\widehat{G}$. Then \[
\intop_{\textrm{Reg}\,\mathcal{\mathcal{C}}}\frac{\overline{\chi}(g)b(x,\xi)}{\left|\det\Phi''(x,\xi,g)_{|N_{(x,\xi,g)}Reg\,\mathcal{\mathcal{C}}}\right|^{1/2}}\, d(\textrm{Reg}\;\mathcal{C})(x,\xi,g)=\left[ \pi_{\chi}|_{H}:\mathds{1}\right]\intop_{\Omega_{\textrm{reg}}}b(x,\xi)\frac{\d\Omega_{\textrm{reg}}(x,\xi)}{\vol_\O\left(G\cdot (x,\xi)\right)}.\]
\end{lem}
\begin{proof}
By using a partition of unity, the proof essentially reduces to the one of  \cite[Lemma 7]{cassanas-ramacher09}, which is based on a result of  \cite[Section 3.3.2]{cassanas}, and  involves only local calculations. Furthermore, $b\in C_{c}^{\infty}(\Omega_{\textrm{reg}})$ is required there. However, similarly as in \cite[Lemma 9.3]{ramacher10}, one can use Fatou's Lemma to show that it suffices to require only $b\in C_{c}^{\infty}(\Omega)$.
\end{proof}

\section{An equivariant semiclassical trace formula}
\label{sec:15.06.2015}
In this section, we generalize the semiclassical trace formula \eqref{eq:result474747433} to an equivariant semiclassical trace formula, which will be crucial for proving the generalized equivariant Weyl law in the next section. As before, let $M$ be  a closed Riemannian manifold of dimension $n$ without boundary, carrying an isometric effective action of a compact connected Lie group $G$ such that the dimension $\kappa$ of the principal orbits is strictly smaller than $n$.  Recall the Peter-Weyl decomposition \eqref{eq:PW} of the left-regular $G$-representation in $\L^2(M)$, and consider a Schr\"odinger operator \eqref{eq:13.08.15} with $G$-invariant potential and Hamiltonian \eqref{eq:13.08.15a}. 

\begin{thm}[\bf Equivariant semiclassical trace formula for Schr\"odinger operators]\label{thm:weakweyl} 
Let $\delta\in \big[0,\frac{1}{2\kappa+3}\big)$,  $\rho_h\in \mathcal{S}^{\mathrm{bcomp}}_{\delta}$, and choose an operator $B\in \Psi_{h,\delta}^0(M)\subset \B(\L^2(M))$ with principal symbol $[b]$ represented by $b\in S_\delta^0(M)$. Consider further for each $h\in (0,1]$ the trace-class operator $$\rho_h(P(h))\circ B:\L^2(M)\to \L^2(M).$$Then, for each semiclassical character family $\{\Wh\}_{h\in(0,1]}$ with growth rate $\vartheta<\frac{1}{2\kappa+3}-\delta$ one has in the semiclassical limit $h\to 0$ the asymptotic formula 
\bq\frac{(2\pi h)^{n-\kappa}}{\#\Wh} \sum_{\chi\in \Wh}\frac{\tr \,\big(\rho_h(P(h))\circ B\big)_\chi}{d_{\chi}\left[ \pi_{\chi}|_{H}:\mathds{1}\right]}
 =\intop_{\Omega_\text{reg}}b\cdot(\rho_h\circ p)\,\frac{\d\Omega_{\textrm{reg}}}{\vol_\O} +\mathrm{O}\Big(h^{1-(2\kappa+3)(\delta+\vartheta)}\left (\log h^{-1}\right)^{{\Lambda^{}}-1}\Big).\label{eq:result338383838}
\eq 
\end{thm}
\begin{rem}
\label{rem:3.2}The integral in the leading term can also be written as
$ 
\intop_{T^*\widetilde{M}_\text{reg}}(\rho_h\circ \widetilde{p}) \, \widetilde{\eklm{b}}_G \d (T^*\widetilde{M}_\text{reg})
$, with notation as in (\ref{eq:orbitalintegral}). To see this, one has to take into account  (\ref{eq:orbithaar}), (\ref{eq:prop81}) and the $G$-invariance of $p$, and apply Lemma \ref{lem:isomorphic} and Corollary \ref{cor:nullset2}. In case that $\widetilde M= M/G$ is an orbifold, the mentioned integral  is given by an integral over the orbifold co-tangent bundle $T^\ast \widetilde M$, see Remark \ref{rem:2.6}. 
\end{rem}
\begin{rem}If $G$ is trivial, the result agrees almost completely with (\ref{eq:result474747433}), the only difference being that the remainder estimate in (\ref{eq:result474747433}) is of order $h^{1-2\delta}$, while (\ref{eq:result338383838}) yields for trivial $G$ (i.e.\ $\kappa=\vartheta=0$, $\Lambda=1$) only the weaker order $h^{1-3\delta}$.
\end{rem}
\begin{proof} Let us consider first  a fixed $\chi\in \widehat G$. Introduce a finite atlas $\{U_\alpha,\gamma_\alpha\}_{\alpha \in \mathcal{A}}$, $\gamma_\alpha: U_\alpha\stackrel{\simeq}{\to} \R^n$, with a subordinated compactly supported partition of unity $\{\varphi_\alpha\}_{\alpha \in \mathcal{A}}$ on $M$  and a family of functions $\{\overline \varphi_\alpha,\overline{\overline \varphi}_\alpha,\overline{\overline{\overline \varphi}}_\alpha\}_{\alpha \in \mathcal{A}}\subset \CT(M)$ such that $\supp \overline \varphi_\alpha,\supp \overline{\overline \varphi}_\alpha,\supp \overline{\overline{\overline \varphi}}_\alpha \subset U_\alpha$ and $\overline \varphi_\alpha\equiv 1$ on $\supp  \varphi_\alpha$, $\overline{\overline \varphi}_\alpha\equiv1$ on $\supp  \overline \varphi_\alpha$, and $\overline{\overline{\overline \varphi}}_\alpha\equiv1$ on $\supp  \overline{\overline \varphi}_\alpha$. For each $\alpha\in\mathcal{A}$, set 
\bqn
u_{\alpha,0}:=\big((\rho_h\circ p)\cdot b\cdot (\varphi_\alpha\circ \tau)\big)\circ(\gamma_\alpha^{-1},(\partial\gamma_\alpha^{-1})^T).
\eqn
Clearly, $u_{\alpha,0}\in S^0_\delta(\R^n)$. Then, by Theorem \ref{thm:semiclassfunccalcmfld}, there is a number $h_0>0$ and for each $N\in \N$ a collection of symbol functions $\{r_{\alpha,\beta,N}\}_{\alpha,\beta\in \mathcal{A}}\subset h^{1-2\delta}S^0_\delta(\R^n)$ and an operator $\mathfrak{R}_{N}(h)\in \B(\L^2(M))$  such that for $h\in (0,h_0]$
\begin{multline*}
\big [T_\chi\circ B\circ \rho_h(P(h))\big ] (f)=\sum_{\alpha \in \mathcal{A}}T_\chi\Big(\overline \varphi_\alpha\cdot\text{Op}_h(u_{\alpha,0})\big((f\cdot \overline{\overline {\overline \varphi}}_\alpha)\circ \gamma_\alpha^{-1}\big)\circ\gamma_\alpha\Big)\\
+\sum_{\alpha,\beta \in \mathcal{A}}T_\chi\Big(\overline \varphi_\beta\cdot\text{Op}_h(r_{\alpha,\beta,N})\big((f\cdot \overline{\overline \varphi}_\alpha\cdot\overline {\overline {\overline \varphi}}_\beta)\circ\gamma_\beta^{-1}\big)\circ\gamma_\beta\Big)  + T_\chi\circ\mathfrak{R}_{N}(h)(f)
\end{multline*}
for all $f\in \L^2(M)$. Moreover, the operator $\mathfrak{R}_N(h)\in \B(\L^2(M))$ is of trace class, its trace norm fulfills
\bqn
\norm{\mathfrak{R}_N(h)}_{\tr,\L^2(M)}=\mathrm{O}\big(h^{N}\big)\quad\text{as }h\to 0,
\eqn
and for fixed $h\in (0,h_0]$ each symbol function $r_{\alpha,\beta,N}$ is an element of $\CT(\R^{2n})$ satisfying 
\bq
\supp  r_{\alpha,\beta,N}\subset \supp   \big((\rho_h\circ p)\cdot(\varphi_\alpha\circ \tau)\big)\circ(\gamma_\alpha^{-1},(\partial\gamma_\alpha^{-1})^T).\label{eq:have798978}
\eq
Inserting the definition \eqref{eq:groupproj} of the projection $T_\chi$ one sees with \eqref{eq:defweyl} that each of the operators 
\begin{align*}
A^\chi_\alpha:f&\mapsto T_\chi\Big(\overline \varphi_\alpha\cdot\text{Op}_h(u_{\alpha,0})\big((f\cdot \overline{\overline {\overline \varphi}}_\alpha)\circ \gamma_\alpha^{-1}\big)\circ \gamma_\alpha\Big),\\
 A^\chi_{\alpha,\beta,N}:f&\mapsto T_\chi\Big(\overline \varphi_\beta\cdot\text{Op}_h(r_{\alpha,\beta,N})\big((f\cdot \overline{\overline \varphi}_\alpha\cdot\overline {\overline {\overline \varphi}}_\beta)\circ\gamma_\beta^{-1}\big)\circ\gamma_\beta\Big)
\end{align*} 
has a smooth, compactly supported Schwartz kernel given respectively by
\begin{align*}
K_{A^\chi_\alpha}(x_1,x_2)&=\frac{d_\chi}{(2\pi h)^n} \int_{G}\int_{\R^n}\overline{\chi(g)}\overline \varphi_\alpha(g^{-1}\cdot x_1)\\
e^{\frac i h \eklm{\gamma_\alpha(g^{-1}\cdot x_1)-\gamma_\alpha(x_2),\eta}}&
 u_{\alpha,0}\Big(\frac{\gamma_\alpha(g^{-1}\cdot x_1)+\gamma_\alpha(x_2)}{2},\eta,h\Big)\overline{\overline {\overline \varphi}}_\alpha(x_2)  \d \eta\d g \big(\text{Vol}_{g_\alpha}(\gamma_\alpha(x_2))\big)^{-1},\\
K_{A^\chi_{\alpha,\beta,N}}(x_1,x_2)&=\frac{d_\chi}{(2\pi h)^n} \int_{G}\int_{\R^n}\overline{\chi(g)}\overline \varphi_\beta(g^{-1}\cdot x_1)\\
e^{\frac i h \eklm{\gamma_\beta(g^{-1}\cdot x_1)-\gamma_\beta(x_2),\eta}}& r_{\alpha,\beta,N}\Big(\frac{\gamma_\beta(g^{-1}\cdot x_1)+\gamma_\beta(x_2)}{2},\eta,h\Big)(\overline{\overline \varphi}_\alpha\cdot\overline {\overline {\overline \varphi}}_\beta)(x_2)  \d \eta\d g\big(\text{Vol}_{g_\beta}(\gamma_\beta(x_2))\big)^{-1},
\end{align*}
where $x_1,x_2\in M$, and $\text{Vol}_{g_\alpha}:\R^n\to (0,\infty)$ denotes the Riemannian volume density function in local coordinates, given by $\text{Vol}_{g_\alpha}(y)=\sqrt{\text{det }g_\alpha(y)}$, $g_\alpha$ being the matrix representing the Riemannian metric on $M$ over the chart $U_\alpha$.  Consequently, we obtain for arbitrary $N\in \N$ that 
\begin{align}
\begin{split}
\tr \,\big (\rho_h(P(h))\circ B\big )_\chi&= \tr \,\big [T_\chi\circ B\circ \rho_h(P(h))\big ]\\ &= \sum_{\alpha \in \mathcal{A}}\int_M K_{A^\chi_\alpha}(x,x)\d M(x)
+\sum_{\alpha,\beta \in \mathcal{A}}\int_M K_{A^\chi_{\alpha,\beta,N}}(x,x)\d M(x)   + \mathrm{O}(h^N),\label{eq:schontrace}
\end{split}
\end{align}
where we took into account that the trace is invariant under cyclic permutations and $T_\chi$ commutes with $\rho_h(P(h))$. Furthermore,   $|\tr Q| \leq \norm{Q}_{\tr,\L^2(M)}$ for any trace class operator $Q$, and 
\bqn
\norm{T_\chi \circ \mathfrak{R}_N(h)}_{\tr,\L^2(M)} \leq \norm{\mathfrak{R}_N(h)}_{\tr,\L^2(M)} \norm{T_\chi}_{\B(\L^2(M))} \leq \norm{\mathfrak{R}_N(h)}_{\tr,\L^2(M)}
\eqn
so that the $\mathrm{O}(h^N)$-estimate is independent of $\chi$. Let us consider first the integrals in the second summand  
\begin{align*}
\int_M K_{A^\chi_{\alpha,\beta,N}}(x,x)\d M(x)=\frac{d_\chi}{(2\pi h)^n} \int_{G}\int_{T^*U_\beta} e^{\frac i h \eklm{\gamma_\beta(x)-\gamma_\beta(g\cdot x),\xi}} u_{\alpha,\beta}^{\chi,N}(x,\xi,g,h)  \d g \d(T^*U_\beta)(x,\xi),
\end{align*}
where  $u_{\alpha,\beta}^{\chi,N}(\cdot,h)\in \CT\left(T^{*}U_{\beta}\times G\right)$ is given by
\begin{align}
\begin{split} u_{\alpha,\beta}^{\chi,N}(x,\xi,g,h)=&\overline{\chi(g)}\overline \varphi_\beta(x)   r_{\alpha,\beta,N}\Big(\frac{\gamma_\beta(x)+\gamma_\beta(g\cdot x)}{2},\xi,h\Big)(\overline{\overline \varphi}_\alpha\cdot\overline {\overline {\overline \varphi}}_\beta)(g\cdot x)J(x,g),\label{eq:u}
\end{split}
\end{align}
$J(x,g)$ being the Jacobian of the substitution $x=g\cdot x '$. By definition of the class $\mathcal{S}^\mathrm{bcomp}_\delta$ there is a compact interval $I\subset \R$  with $\supp \rho_h\subset I$ for all $h\in (0,1]$. Taking into account (\ref{eq:have798978}) and the definition (\ref{eq:13.08.15a}) of $p$ we see that the function $u_{\alpha,\beta}^{\chi,N}(\cdot,h)$ is supported inside a compact $h$-independent subset of $T^{*}U_{\beta}\times G$. Theorem \ref{thm:main} now implies for each $N\in\mathbb{N}$ the estimate
\begin{align}
\nonumber \bigg| (2\pi h)^{n}\int_M  & K_{A^\chi_{\alpha,\beta,N}}(x,x)\d M(x) 
- d_{\chi}\left(2\pi h\right)^{\kappa}{\intop_{\textrm{Reg}\,{\mathcal{C}}_\beta}\frac{u_{\alpha,\beta}^{\chi,N}(x,\xi,g,h)}{\left|\det\Phi''(x,\xi,g)_{|N_{(x,\xi,g)}Reg\,\mathcal{\mathcal{C}}_{\beta}}\right|^{1/2}}\, d(\textrm{Reg}\,{\mathcal{C}_\beta})(x,\xi,g)}\bigg|\\
&\leq C_{\alpha,\beta,N} \, d_\chi \sup_{l \leq 2\kappa+3} \big \|D^l u_{\alpha,\beta}^{\chi,N}\big \|_{\infty}h^{\kappa+1}\left (\log h^{-1}\right)^{{\Lambda^{}}-1},\label{eq:firsttrace}
\end{align}
where $\textrm{Reg}\,\mathcal{\mathcal{C}}_{\beta}=\{(x,\xi,g)\in(\Omega\cap T^{*}U_{\beta})\times G,\; g\cdot (x,\xi)=(x,\xi),\;x\in M(H)\}$, $D^l$ is a differential operator of order $l$, 
and $\Phi''(x,\xi,g)_{|N_{(x,\xi,g)}\textrm{Reg }\mathcal{C}_{\beta}}$ denotes the restriction of the Hessian of $\Phi(x,\xi,g)=\left\langle \gamma_{\beta}(x)-\gamma_{\beta}(g\cdot x ),\xi\right\rangle$ to the normal space of $\mathrm{Reg}\,\mathcal{\mathcal{C}}_{\beta}$ inside $T^{*}U_{\beta}\times G$ at the point $(x,\xi,g)$.  Note that the domain of integration of the integral
\bqn 
\mathfrak{A}^\chi_{\alpha,\beta,N}(h):=\intop_{\textrm{Reg}\,{\mathcal{C}}_\beta}\frac{u_{\alpha,\beta}^{\chi,N}(x,\xi,g,h)}{\left|\det\Phi''(x,\xi,g)_{|N_{(x,\xi,g)}Reg\,\mathcal{\mathcal{C}}_{\beta}}\right|^{1/2}}\, d(\textrm{Reg}\,{\mathcal{C}_\beta})(x,\xi,g)
\eqn
 contains only such $g$ and $x$ for which  $g\cdot x =x$, so that it simplifies to 
\begin{align}
\begin{split}
\mathfrak{A}^\chi_{\alpha,\beta,N}(h)=\intop_{\textrm{Reg}\,{\mathcal{C}_\beta}}\frac{\overline{\chi(g)} r_{\alpha,\beta,N}(\gamma_\beta(x),\xi,h)(\overline{\overline \varphi}_\alpha\cdot\overline \varphi_\beta)(x)}{\left|\det\Phi''(x,\xi,g)_{|N_{(x,\xi,g)}Reg\,{\mathcal{C}}_{\beta}}\right|^{1/2}}\, d(\textrm{Reg}\,\mathcal{\mathcal{C}_\beta})(x,\xi,g). \label{eq:integrand1111}
\end{split}
\end{align}
Here we used that $J(x,g)=1$ in the domain of integration, 
since the substitution $x'=g\cdot x $ is the identity when $g\cdot x =x$, and that $\overline{ \overline { \overline \varphi}}_{\beta}\equiv1$ on $\supp \overline \varphi_{\beta}$. By Lemma \ref{thm:integralconcrete} this simplifies further to
\begin{equation*}
\mathfrak{A}^\chi_{\alpha,\beta,N}(h)=\left[ \pi_{\chi}|_{H}:\mathds{1}\right]\intop_{\Omega_{\textrm{reg}}}r_{\alpha,\beta,N}(\gamma_\beta(x),\xi,h) (\overline{\overline \varphi}_\alpha\cdot\overline \varphi_\beta)(x)\frac{\d\Omega_{\textrm{reg}}(x,\xi)}{\vol\left(G\cdot (x,\xi)\right)}.
\end{equation*}
We obtain that there is a constant $C_{\alpha,\beta,N}>0$, independent of $h$ and $\chi$, such that
\bqn
\big|\mathfrak{A}^\chi_{\alpha,\beta,N}(h)\big|
 \leq C_{\alpha,\beta,N} \left[ \pi_{\chi}|_{H}:\mathds{1}\right]\norm{r_{\alpha,\beta,N}}_\infty \norm{(\overline{\overline \varphi}_\alpha\cdot\overline \varphi_\beta)}_\infty.
\eqn
As $r_{\alpha,\beta,N}$ is an element of $h^{1-2\delta}S^0_\delta(\R^n)$ we have that  $\norm{r_{\alpha,\beta,N}}_\infty=\mathrm{O}(h^{1-2\delta})$, so that we conclude
\[
\frac{\big|\mathfrak{A}^\chi_{\alpha,\beta,N}(h)\big|}{\left[ \pi_{\chi}|_{H}:\mathds{1}\right]} =\mathrm{O}\big(h^{1-2\delta}\big)\quad \text{as }h\to 0,
\]
the estimate being independent of $\chi$. Now, we combine this result with (\ref{eq:firsttrace}) and  the relation $r_{\alpha,\beta,N}\in h^{1-2\delta}S^0_\delta(\R^n)$ to obtain the estimate
\begin{multline}
\Big|\int_M K_{A^\chi_{\alpha,\beta,N}}(x,x)\d M(x)\Big|
=\mathrm{O}\Big(d_\chi\left[ \pi_{\chi}|_{H}:\mathds{1}\right]h^{1-2\delta-n+\kappa}\;+\\ d_\chi \, \sup_{l \leq 2\kappa+3} \big \|\widetilde{D}^l \overline \chi\big \|_{\infty}h^{-n+\kappa+1-\delta(2\kappa+3)+1-2\delta}\left (\log h^{-1}\right)^{{\Lambda^{}}-1}\Big),\label{eq:estim7373737373}
\end{multline}
where $\widetilde{D}^l$ is a differential operator of order $l$ on $G$ and the constant in the estimate is independent of $\chi$. Since $\mathcal{A}$ is finite, we conclude from (\ref{eq:schontrace}) and (\ref{eq:estim7373737373}) that
\begin{align}
\begin{split}
 \tr \,\big (\rho_h(P(h))\circ B \big )_\chi  &= \sum_{\alpha \in \mathcal{A}}\int_M K_{A^\chi_\alpha}(x,x)\d M(x)\\
&+ \mathrm{O}\Big(h^{1-2\delta-n+\kappa}d_\chi\Big[\left[ \pi_{\chi}|_{H}:\mathds{1}\right]+ \sup_{l \leq 2\kappa+3} \big \|\widetilde{D}^l \overline \chi\big \|_{\infty}h^{1-\delta(2\kappa+3)}\left (\log h^{-1}\right)^{{\Lambda^{}}-1}\Big]\Big),\label{eq:trace382958239}
\end{split}
\end{align}
with the constant in the estimate being independent of $\chi$. Let us now calculate the integrals in the leading term. Again, we can apply Theorem \ref{thm:main}, and by steps  analogous to those above we arrive at
\begin{align*}
\bigg| (2\pi h)^{n}\int_M &K_{A^\chi_\alpha}(x,x)\d M(x) 
- d_{\chi}\left(2\pi h\right)^{\kappa}\left[ \pi_{\chi}|_{H}:\mathds{1}\right]\intop_{\Omega_{\textrm{reg}}}u_{\alpha,0}(\gamma_\alpha(x),\xi,h) \overline \varphi_\alpha(x)\frac{\d\Omega_{\textrm{reg}}(x,\xi)}{\vol\left(G\cdot (x,\xi)\right)}\bigg|\\
&\leq C_{\alpha} d_\chi \, \sup_{l' \leq 2\kappa+3} \big \|\widetilde{D}^{l'} \overline \chi\big \|_{\infty} \sup_{l \leq 2\kappa+3} \big \|\hat D^l u_{\alpha,0}\big \|_{\infty}h^{\kappa+1}\left (\log h^{-1}\right)^{{\Lambda^{}}-1},
\end{align*}
where $C_\alpha$ is independent of $h$ and $\chi$, and $\hat D^l$ is a differential operator on $\R^{2n}$ of order $l$. Since $\rho_h$ is an element of $\mathcal{S}^{\mathrm{bcomp}}_{\delta}$ one has $\sup_{l \leq 2\kappa+3} \big \|\hat D^l u_{\alpha,0}\big \|_{\infty}=\mathrm{O}(h^{-(2\kappa+3)\delta})$, yielding the estimate 
\begin{align*}
\bigg| (2\pi h)^{n}\int_M K_{A^\chi_\alpha}(x,x)\d M(x) &
- d_{\chi}\left(2\pi h\right)^{\kappa}\left[ \pi_{\chi}|_{H}:\mathds{1}\right]\intop_{\Omega_{\textrm{reg}}}\big((\rho_h\circ p)\cdot b\cdot \varphi_\alpha\big)(x,\xi)\frac{\d\Omega_{\textrm{reg}}(x,\xi)}{\vol\left(G\cdot (x,\xi)\right)}\bigg|\\
&=\mathrm{O}\Big(h^{1+\kappa-(2\kappa+3)\delta} \, d_\chi \, \sup_{l \leq 2\kappa+3} \big \|\widetilde{D}^l \overline \chi\big \|_{\infty}\left (\log h^{-1}\right)^{{\Lambda^{}}-1}\Big)
\end{align*}
as $h\to 0$. Summing over the finite set $\mathcal{A}$, and using (\ref{eq:trace382958239}) together with $\overline \varphi_\alpha\equiv 1$ on $\supp  \varphi_\alpha$ and $\sum_{\alpha\in \mathcal{A}}\varphi_\alpha=1$ we finally obtain
\begin{multline}
\left(2\pi h\right)^{n-\kappa}\frac{\tr \,\big (\rho_h(P(h))\circ B \big )_\chi}{d_{\chi}\left[ \pi_{\chi}|_{H}:\mathds{1}\right]} 
= \intop_{\Omega_{\textrm{reg}}}\big((\rho_h\circ p)\cdot b\big)(x,\xi)\frac{\d\Omega_{\textrm{reg}}(x,\xi)}{\vol\left(G\cdot (x,\xi)\right)}\label{eq:resultfixedchi}\\
+\mathrm{O}\Big(h^{1-2\delta}+W_\kappa(\chi)h^{1-\delta(2\kappa+3)}\left (\log h^{-1}\right)^{{\Lambda^{}}-1}\Big),
\end{multline}
where the constant in the estimate is independent of $\chi$ and we introduced the notation
\bq
\label{eq:14.08.15}
W_\kappa(\chi):=\frac{\sup_{l \leq 2\kappa+3} \big \|\widetilde{D}^l \overline \chi\big \|_{\infty}}{\left[ \pi_{\chi}|_{H}:\mathds{1}\right]}.
\eq
Now, having established the result (\ref{eq:resultfixedchi}) for a fixed $\chi$, we know precisely how the remainder estimate depends on $\chi$ and we see that the leading term is independent of $\chi$. Thus, we can average for each $h\in (0,1]$ each summand in (\ref{eq:resultfixedchi}) over the finite set $\Wh$ to obtain the result
 \begin{multline*}
\frac{\left(2\pi h\right)^{n-\kappa}}{\#\Wh}\sum_{\chi\in\Wh}\frac{\tr \,\big (\rho_h(P(h))\circ B \big )_\chi}{d_{\chi}\left[ \pi_{\chi}|_{H}:\mathds{1}\right]} 
= \intop_{\Omega_{\textrm{reg}}}\big((\rho_h\circ p)\cdot b\big)(x,\xi)\frac{\d\Omega_{\textrm{reg}}(x,\xi)}{\vol\left(G\cdot (x,\xi)\right)}\\
+\mathrm{O}\Big(h^{1-2\delta}+\Big[\frac 1 {\# \W_h} \sum_{\chi\in \Wh}W_\kappa(\chi)\Big]h^{1-\delta(2\kappa+3)}\left (\log h^{-1}\right)^{{\Lambda^{}}-1}\Big).
\end{multline*}
To finish the proof, it suffices to observe that since  the growth rate of the family $\{\Wh\}_{h\in(0,1]}$ is $\vartheta$ we have $\frac 1 {\# \W_h} \sum_{\chi\in \Wh} W_\kappa(\chi)=\mathrm{O}\big(h^{-(2\kappa+3)\vartheta}\big)$ as $h\to 0$, and the assertion (\ref{eq:result338383838}) follows. 
\end{proof}

\section{A generalized equivariant semiclassical Weyl law}

Let the notation be as in the previous sections. We are now in the position to state and prove the main result of this paper.

\begin{thm}[\bf Generalized equivariant semiclassical Weyl law]
\label{thm:weyl2} 
Let $\delta\in \big(0,\frac{1}{2\kappa+4}\big)$ and choose an operator $B\in \Psi_{h,\delta}^0(M)\subset \B(\L^2(M))$ with principal symbol represented by $b\in S_\delta^0(M)$ and a semiclassical character family $\{\Wh\}_{h\in(0,1]}$ with growth rate $\vartheta<\frac{1-(2\kappa+4)\delta}{2\kappa+3}$.  Write
\[
J(h):=\big\{j\in \N:E_j(h)\in[c,c+h^\delta],\; \chi_j(h) \in \Wh\big\},
\]
where $\chi_j(h)\in \widehat G$ is defined by $u_j(h)\in \L^2_{\chi_j(h)}(M)$. Then, one has in the semiclassical limit $h\to 0$ 
\begin{align*}
\begin{split}
\frac{(2\pi)^{n-\kappa} h^{n-\kappa-\delta}}{\#\Wh}\sum_{J(h)}\frac{\langle Bu_{j}(h),u_{j}(h)\rangle_{\L^2(M)}}{d_{\chi_j(h)}\,[ \pi_{\chi_j(h)}|_{H}:\mathds{1}]}&=\intop_{{\Sigma}_c\cap \,\Omega_{\text{reg}}}b \, \frac {\d{\mu}_c}{\vol_\O}+\; \mathrm{O}\Big(h^{\delta}+h^{\frac{1-(2\kappa+3)\vartheta}{2\kappa +4}-\delta}\left (\log h^{-1}\right)^{{\Lambda^{}}-1}\Big).
\end{split}
\end{align*}
\end{thm}
\begin{rem}\label{rem:0304}
Note that the  second summand in the remainder dominates the estimate if and only if $\delta\geq \frac{1-(2\kappa+3)\vartheta}{4\kappa +8}$. 
Further, the integral in the leading term  equals  $\intop_{\widetilde{\Sigma}_c}\widetilde {\eklm{b}}_G\, \d\widetilde{\Sigma}_c$, compare Section \ref{subsec:spaces}, and it can actually be viewed as an integral over the smooth bundle 
\bqn
S^\ast_{\widetilde p,c}(\widetilde M_\text{reg}):=\mklm{(x,\xi) \in T^\ast (\widetilde M_\text{reg}): \widetilde p(x,\xi)=c},
\eqn
where $\widetilde p$ is the function on $T^\ast (\widetilde M_\text{reg})$ induced by $p$  via Lemma \ref{lem:isomorphic}.  In case that $\widetilde M$ is an orbifold, the mentioned integral is given by an integral over the orbifold bundle $S^\ast_{\widetilde p,c}(\widetilde M):=\mklm{(x,\xi) \in T^\ast \widetilde M: \widetilde p(x,\xi)=c}$, compare Remark \ref{rem:3.2}.
\end{rem}
In the special case of a constant semiclassical character family, corresponding to the study of a single fixed isotypic component, we obtain as a direct corollary
\begin{thm}\label{thm:weyl31}Choose a fixed $\chi\in \widehat G$.  Then  for each $\delta\in(0,\frac{1}{2\kappa +4})$ one has the asymptotic formula
\begin{align}
\label{eq:intBweak}
\begin{split}
(2\pi)^{n-\kappa} h^{n-\kappa-\delta}\hspace*{-1.75em}\sum_{\begin{array}{c}\scriptstyle 
j\in\mathbb{N}:\,  u_{j}(h)\in \L_{\chi}^{2}(M),\\
\scriptstyle E_{j}(h)\in [c,c+h^\delta]\end{array}}\hspace*{-1.5em}&\langle Bu_{j}(h),u_{j}(h)\rangle_{\L^2(M)}=d_{\chi}\,[ \pi_{\chi}|_{H}:\mathds{1}]\intop_{{\Sigma}_c\cap \,\Omega_{\text{reg}}}b \, \frac {\d{\mu}_c}{\vol_\O}\\
 &+\; \mathrm{O}\Big(h^{\delta}+h^{\frac{1}{2\kappa+4}-\delta}\left (\log h^{-1}\right)^{\Lambda^{}-1}\Big), \qquad h \to 0.
\end{split}
\end{align}
\end{thm}

\begin{rem}
\label{rem:4.2}
A weaker version of Theorem \ref{thm:weyl31} can be proved if instead of the spectral window $[c,c+h^\delta]$ one considers a fixed interval $[r,s]$, the numbers $r,s$ being regular values of $p$. One can then show that
\[
(2\pi h)^{n-\kappa} \hspace*{-1em}\sum_{\begin{array}{c}\scriptstyle 
j\in\mathbb{N}:\,  u_{j}(h)\in \L_{\chi}^{2}(M),\\
\scriptstyle E_{j}(h)\in [r,s]\end{array}}\hspace*{-1em}\frac{\left\langle Bu_{j}(h),u_{j}(h)\right\rangle_{\L^2(M)}}{d_{\chi}\,[ \pi_{\chi}|_{H}:\mathds{1}]}=\intop_{p^{-1}([r,s])\cap\Omega_\text{reg}}b\,\frac{d\Omega_\text{reg}}{\vol_\O}
\;+\;\mathrm{O}\left(h^{\frac{1}{2\kappa+4}} (\log h^{-1})^{{\Lambda^{}}-1}\right),
\]
which is proven in complete analogy. The even weaker statement
\bqn 
\lim _{h \to 0} (2\pi h)^{n-\kappa} \hspace*{-1em}\sum_{\begin{array}{c}\scriptstyle 
j\in\mathbb{N}:\,  u_{j}(h)\in \L_{\chi}^{2}(M),\\
\scriptstyle E_{j}(h)\in [r,s]\end{array}}\hspace*{-1em}\frac{\left\langle Bu_{j}(h),u_{j}(h)\right\rangle_{\L^2(M)}}{d_{\chi}\,[ \pi_{\chi}|_{H}:\mathds{1}]}=\intop_{p^{-1}([r,s])\cap\Omega_\text{reg}}b\,\frac{d\Omega_\text{reg}}{\vol_\O}
\eqn
could in principle  also be obtained without  the remainder estimates from \cite{ramacher15a} using heat kernel methods as in \cite{donnelly78} or \cite{bruening-heintze79}, adapted to the semiclassical setting. Nevertheless, for the study of growing families of isotypic components and shrinking spectral windows as in Theorem \ref{thm:weyl2} remainder estimates are necessary due to the lower rate of convergence. 
\end{rem}

\begin{proof}[Proof of Theorem \ref{thm:weyl2}]
The proof is an adaptation of the proof of \cite[Theorem 15.3]{zworski}
to our situation, but with a sharper energy localization. Again, we consider first  a single character $\chi\in \widehat G$. Let $h\in (0,1]$ and fix a positive number $\lambda<\frac{1}{2\kappa+3}-\delta$. Choose $f_{\lambda,h},g_{\lambda,h}\in \CT(\mathbb{R},[0,1])$ such that $\supp f_{\lambda,h}\subset[-\frac{1}{2}+{h^\lambda},\frac{1}{2}-{h^\lambda}]$,
$f_{\lambda,h}\equiv1$ on $[-\frac{1}{2}+3{h^\lambda},\frac{1}{2}-3{h^\lambda}]$, $\supp g_{\lambda,h}\subset[-\frac{1}{2}-3{h^\lambda},\frac{1}{2}+3{h^\lambda}]$,
$g_{\lambda,h}\equiv1$ on $[-\frac{1}{2}-{h^\lambda},\frac{1}{2}+{h^\lambda}]$, and 
\begin{equation}
\label{eq:derivat}
|\gd_y^j f_{\lambda,h}(y)| \leq C_j \, {h^{-\lambda j}}, \qquad |\gd_y^j g_{\lambda,h}(y)| \leq C_j \, {h^{-\lambda j}},
\end{equation}
compare \cite[Theorem 1.4.1 and (1.4.2)]{hoermanderI}. Put  $c(h):=ch^{-\delta}+\frac{1}{2}$, so that $x\mapsto h^{-\delta}x -c(h)$ defines a diffeomorphism from $[c,c+h^\delta]$ to  $[-1/2,1/2]$, and set $f_{\lambda,\delta,h}(x):=f_{\lambda,h}(h^{-\delta}x-c(h))$, 
$g_{\lambda,\delta,h}(x):=g_{\lambda,h}(h^{-\delta}x-c(h))$.
Let $\Pi_{\chi}$ be the projection onto the span of $\{u_{j}(h)\in \L_{\chi}^{2}(M):\; E_{j}(h)\in[c,c+h^\delta]\}$. Then
\begin{align}
\begin{split}
f_{\lambda,\delta,h}(P(h))_{\chi}\circ \Pi_{\chi} & =\Pi_{\chi}\, \circ f_{\lambda,\delta,h}(P(h))_{\chi}=f_{\lambda,\delta,h}(P(h))_{\chi},\label{eq:equal}\\
g_{\lambda,\delta,h}(P(h))_{\chi}\circ\Pi_{\chi} & =\Pi_{\chi}\, \circ g_{\lambda,\delta,h}(P(h))_{\chi}=\Pi_{\chi}.
\end{split}
\end{align}
Note that the operators $f_{\lambda,\delta,h}(P(h))$, $g_{\lambda,\delta,h}(P(h))$, $\Pi_{\chi}$ are finite rank operators. For that elementary
reason, all operators we consider in the following are trace class.
In particular,  by (\ref{eq:equal}) we have
\[\hspace*{-6.75cm}
\sum_{\begin{array}{c}\scriptstyle 
j\in\mathbb{N}:\,  u_{j}(h)\in \L_{\chi}^{2}(M) \\
\scriptstyle E_{j}(h)\in[c,c+h^{\delta}]\end{array}}\left\langle Bu_{j}(h),u_{j}(h)\right\rangle_{\L^2(M)}\; =\;\textrm{tr }\Pi_{\chi}\,\circ B\,\circ\,\Pi_{\chi}
\]\vspace*{-3em}
\begin{eqnarray}
\hspace*{2.7cm}& \overset{\textrm{}}{=} & \textrm{tr }f_{\lambda,\delta,h}(P(h))_{\chi}\circ B_{\chi}+\textrm{tr }\Pi_{\chi}\circ g_{\lambda,\delta,h}(P(h))_{\chi}\circ \left(1-f_{\lambda,\delta,h}(P(h))_{\chi}\right)\circ B_{\chi}\circ\Pi_{\chi}\nonumber \\
 & = & \textrm{tr }\big(f_{\lambda,\delta,h}(P(h))\circ B\big)_{\chi}+\underbrace{\textrm{tr }\Pi_{\chi}\circ \big(g_{\lambda,\delta,h}(P(h))\circ \left(1-f_{\lambda,\delta,h}(P(h))\right)\circ B\big)_{\chi} \circ\Pi_{\chi}}_{=\mathfrak{R}_{\lambda,\delta,h}}.\label{eq:Bsum}
\end{eqnarray}
In what follows, we shall show that the first summand in \eqref{eq:Bsum} represents the main contribution, while $\mathfrak{R}_{\lambda,\delta,h}$ becomes small as $h$ goes to zero. For this, we estimate $\mathfrak{R}_{\lambda,\delta,h}$ using the trace norm. Recall that if $L\in \B(\L^2(M))$ is of trace class and $M\in \B(\L^2(M))$, then $\left\Vert LM\right\Vert_{\tr ,{\L}^{2}(M)}\leq\left\Vert L\right\Vert _{\tr ,{\L}^{2}(M)}\left\Vert M\right\Vert _{\B\left(\L^2(M)\right)}$, see e.g.\ \cite[p.\ 337]{zworski}. By the spectral functional calculus this implies
\begin{align}
\begin{split}
 |\mathfrak{R}_{\lambda,\delta,h}| & \leq  \left\Vert \Pi_{\chi}\circ \left(g_{\lambda,\delta,h}(P(h))\circ \left(1-f_{\lambda,\delta,h}(P(h))\right)\circ B\right)_{\chi}\circ \Pi_{\chi}\right\Vert _{\tr ,{\L}^{2}(M)} \\
 & \leq  \left\Vert \left(g_{\lambda,\delta,h}(P(h))\circ \left(1-f_{\lambda,\delta,h}(P(h))\right)\right)_{\chi}\right\Vert _{\tr ,{\L}^{2}(M)}\left\Vert B\right\Vert _{B\left(\L^2(M)\right)} \\
 & = \left\Vert v_{\lambda,\delta,h}(P(h))_{\chi}\right\Vert _{\tr ,{\L}^{2}(M)}\left\Vert B\right\Vert _{B\left(\L^2(M)\right)},\label{eq:normB}
 \end{split}
\end{align}
where we set $v_{\lambda,\delta,h}=g_{\lambda,\delta,h}(1-f_{\lambda,\delta,h})\in \CT(\mathbb{R},[0,1])$.
In particular, $v_{\lambda,\delta,h}$ is non-negative. By the spectral functional calculus,
$v_{\lambda,\delta,h}(P(h))$ is a positive operator. $T_{\chi}$ is a
projection, hence positive as well. Moreover, by the spectral functional calculus, $v_{\lambda,\delta,h}(P(h))_{\chi}$ commutes with $T_\chi$, as $P(h)$ does. It follows that $v_{\lambda,\delta,h}(P(h))_{\chi}$
is positive as the composition of positive commuting operators. For a positive
operator, the trace norm is identical to the trace. Therefore (\ref{eq:normB})
implies
\begin{equation}
|\mathfrak{R}_{\lambda,\delta,h}|\leq \left\Vert B\right\Vert _{B\left(\L^2(M)\right)} \, \textrm{tr }v_{\lambda,\delta,h}(P(h))_{\chi}.\label{eq:R}
\end{equation}
From our knowledge about the supports of $f_{\lambda,h}$ and $g_{\lambda,h}$,
we conclude that 
\begin{equation}
\supp v_{\lambda,\delta,h}\subset[c-3{h^\lambda} h^{\delta},c+3{h^\lambda} h^{\delta}]\cup[c+h^{\delta}-3{h^\lambda} h^{\delta},c+h^{\delta}+3{h^\lambda} h^{\delta}].\label{14.08}
\end{equation}
Now, note that the functions $f_{\lambda,\delta,h}, \,g_{\lambda,\delta,h},\,v_{\lambda,\delta,h}$ are elements of $\mathcal{S}^{\mathrm{bcomp}}_{\lambda+\delta}$. Since we chose $\lambda$ such that $\lambda+\delta<\frac{1}{2\kappa+3}$,  we can apply (\ref{eq:resultfixedchi}) with $B= \id_{\L^2(M)}$ to conclude
\begin{multline}\left | \left(2\pi h\right)^{n-\kappa} \frac{\tr v_{\lambda,\delta,h}\left(P(h)\right)_{\chi}}{d_{\chi}\left[ \pi_{\chi}|_{H}:\mathds{1}\right]}-\int_{{\Omega}_{\textrm{reg}}}(v_{\lambda,\delta,h}\circ {p})\,  \frac{ d{\Omega}_{\textrm{reg}}}{\vol_\O}\right | \\
\leq C\Big(h^{1-2(\lambda+\delta)}+W_\kappa(\chi)h^{1-(\lambda+\delta)(2\kappa+3)}\left (\log h^{-1}\right)^{{\Lambda^{}}-1}\Big),\label{eq:4.7}
\end{multline}
where $C$ is independent of $h$ and $\chi$, and $W_\kappa(\chi)$
was defined in \eqref{eq:14.08.15}. On the other hand, applying (\ref{eq:resultfixedchi}) to the first summand on the right hand side of (\ref{eq:Bsum}) yields
\begin{multline}
\left | \left(2\pi h\right)^{n-\kappa} \frac{\tr \left(f_{\lambda,\delta,h}\left(P(h)\right)\circ B\right)_{\chi}}{d_{\chi}\left[ \pi_{\chi}|_{H}:\mathds{1}\right]}-\int_{{\Omega}_{\textrm{reg}}}(f_{\lambda,\delta,h}\circ {p})\, b \, \frac{ d{\Omega}_{\textrm{reg}}}{\vol_\O}\right |\\
\leq C \Big(h^{1-2(\lambda+\delta)}+W_\kappa(\chi)h^{1-(\lambda+\delta)(2\kappa+3)}\left (\log h^{-1}\right)^{{\Lambda^{}}-1}\Big),\label{eq:4.8989}
\end{multline}
where $C$ is a new constant independent of $h$ and $\chi$. Next, observe that the supports of the functions $v_{\lambda,\delta,h}\circ p$ and $v_{\lambda,\delta,h}\circ \widetilde p$ are contained in tubular neighbourhoods of width of order $h^{\delta+\lambda}$ around hypersurfaces of $T^*M$ resp.\ $T^*\widetilde{M}_\text{reg}$, and Lemma \ref{lem:hypersurf} and Corollary \ref{cor:hypersurf2} imply 
\bq
\Big|\int_{{\Omega}_{\textrm{reg}}}(v_{\lambda,\delta,h}\circ{p}) \, \frac{ d{\Omega}_{\textrm{reg}}}{\vol_\O}\Big|=\Big|\int_{T^*\widetilde{M}_{\textrm{reg}}}(v_{\lambda,\delta,h}\circ\widetilde {p}) \d(T^*\widetilde{M}_{\textrm{reg}})\Big|=\mathrm{O}(h^{\delta+\lambda})\qquad \text{as }h\to 0.\label{eq:supportshrink2}
\eq
Combining \eqref{eq:Bsum}-\eqref{eq:supportshrink2} leads to
\begin{multline}
\begin{split}
\left(2\pi h\right)^{n-\kappa}\sum_{\begin{array}{c}\scriptstyle 
j\in\mathbb{N}:\,  u_{j}(h)\in \L_{\chi}^{2}(M)\\
\scriptstyle E_{j}(h)\in [c,c+h^\delta]\end{array}}\frac{\left\langle Bu_{j}(h),u_{j}(h)\right\rangle_{\L^2(M)}}{d_{\chi}\left[ \pi_{\chi}|_{H}:\mathds{1}\right]}=\intop_{\;{\Omega}_{\textrm{reg}}}(f_{\lambda,\delta,h}\circ {p}) \, b\, \frac{ d{\Omega}_{\textrm{reg}}}{\vol_\O}\\
+\;\mathrm{O}\Big(h^{\delta+\lambda}+h^{1-2(\lambda+\delta)}+W_\kappa(\chi)h^{1-(\lambda+\delta)(2\kappa+3)}\left (\log h^{-1}\right)^{{\Lambda^{}}-1}\Big),\label{eq:good}
\end{split}
\end{multline}
the constant in the estimate being independent of $\chi$. We proceed by observing
\begin{align}
\begin{split}
\Big |\int_{{\Omega}_{\textrm{reg}}}(f_{\lambda,\delta,h}\circ {p})\, b \, \frac{ d{\Omega}_{\textrm{reg}}}{\vol_\O} &-\int_{{\Omega}_{\textrm{reg}}\cap p^{-1}([c,c+h^\delta])} b \, \frac{ d{\Omega}_{\textrm{reg}}}{\vol_\O}\Big | \\
&\leq \Big | \int_{T^*\widetilde{M}_{\textrm{reg}}}(v_{\lambda,\delta,h}\circ\widetilde {p})\, b \d(T^*\widetilde{M}_{\textrm{reg}})\Big| = \mathrm{O}(h^{\lambda+\delta}).
\end{split}
\end{align}
Furthermore, with ${\Sigma}_c={p}^{-1}(\{c\})$,  $\widetilde{\Sigma}_c=\widetilde{p}^{-1}(\{c\})$,  and the notation from (\ref{eq:orbitalintegral}) one computes
\begin{align}
\label{eq:4.11}
\begin{split}
\frac{1}{h^\delta}\int_{{\Omega}_{\textrm{reg}}\cap p^{-1}([c,c+h^\delta])} b \, \frac{ d{\Omega}_{\textrm{reg}}}{\vol_\O} &=\frac{1}{h^\delta}\int_{T^*\widetilde{M}_{\textrm{reg}}\cap \widetilde{p}^{-1}([c,c+h^\delta])} \widetilde{\eklm{b}}_G \d T^*\widetilde{M}_{\textrm{reg}}\\
&= \int_{\widetilde{\Sigma}_c}\widetilde{\eklm{b}}_G \d \widetilde{\Sigma}_c +\mathrm{O}(h^\delta) =  \intop_{{\Sigma}_c\cap \,\Omega_{\text{reg}}}b \, \frac {\d{\mu}_c}{\vol_\O}+\mathrm{O}(h^\delta),
\end{split}
\end{align}
where  we took into account Lemmas \ref{lem:isomorphic} and Corollary \ref{cor:hypersurf2}. Combining (\ref{eq:good})-(\ref{eq:4.11}) then yields for a fixed $\chi \in \widehat G$
\begin{align}
\begin{split}
(2\pi)^{n-\kappa}h^{n-\kappa-\delta}\sum_{\begin{array}{c}\scriptstyle 
j\in\mathbb{N}:\, u_{j}(h)\in \L_{\chi}^{2}(M) \\
\scriptstyle E_{j}(h)\in [c,c+h^\delta]\end{array}}&\frac{\left\langle Bu_{j}(h),u_{j}(h)\right\rangle_{\L^2(M)}}{d_{\chi}\left[ \pi_{\chi}|_{H}:\mathds{1}\right]}-\intop_{{\Sigma}_c\cap \,\Omega_{\text{reg}}}b \, \frac {\d{\mu}_c}{\vol_\O} \\
&=\mathrm{O}\Big(h^{\delta}+h^\lambda+W_\kappa(\chi)h^{1-(\lambda+\delta)(2\kappa+3)-\delta}\left (\log h^{-1}\right)^{{\Lambda^{}}-1}\Big).\label{eq:remainderterms78}
\end{split}
\end{align}
Here, the constant in the estimate is independent of $\chi$. Just as at the end of the proof of Theorem \ref{thm:weakweyl}, we can now take for each $h\in(0,1]$ the average over the finite set $\Wh$, and knowing that $\Wh$ has growth rate $\vartheta$, we get
\begin{align}
\begin{split}
\frac{(2\pi)^{n-\kappa}h^{n-\kappa-\delta}}{\#\Wh}&\sum_{J(h)}\frac{\left\langle Bu_{j}(h),u_{j}(h)\right\rangle_{\L^2(M)}}{d_{\chi_j(h)}\left[ \pi_{\chi_j(h)}|_{H}:\mathds{1}\right]}\;-\intop_{{\Sigma}_c\cap \,\Omega_{\text{reg}}}b \, \frac {\d{\mu}_c}{\vol_\O} \\
&=\mathrm{O}\Big(h^{\delta}+h^\lambda+h^{1-(\lambda+\delta+\vartheta)(2\kappa+3)-\delta}\left (\log h^{-1}\right)^{{\Lambda^{}}-1}\Big).\label{eq:averagedresult}
\end{split}
\end{align}
Finally, we choose $\lambda$ such that the remainder estimate is optimal for the given constants $\delta$ and $\vartheta$. This is the case iff
$\lambda=1-(\lambda+\delta+\vartheta)(2\kappa+3)-\delta$, which is equivalent to
\[
\lambda=\frac{1-(2\kappa+3)\vartheta}{2\kappa +4}-\delta.
\]
This choice for $\lambda$ is compatible with the general technical requirement that $\lambda<\frac{1}{2\kappa+3}-\delta$, and the assertion of Theorem \ref{thm:weyl2} follows.
\end{proof}

As consequence of the previous theorem we obtain in particular

\begin{thm}[\bf Equivariant Weyl law for semiclassical character families]
\label{thm:weyl1}For each $\chi \in \widehat G$, denote by $\mathrm{mult}_\chi(E_j(h))$ the multiplicity of the irreducible representation $\pi_\chi$ in the eigenspace $\E_j(h)$ corresponding to the eigenvalue $E_j(h)$. Then  one has in the limit $h\to 0$ the asymptotic formula
\bqn
\frac{(2\pi)^{n-\kappa} h^{n-\kappa-\delta}}{\#\Wh}\hspace{-1.5em}\sum_{\begin{array}{c}\scriptstyle 
\chi\in \Wh,\;j\in\mathbb{N}:\\
\scriptstyle E_{j}(h)\in\, [c,c+h^\delta]\end{array}}\hspace{-1em} \frac{\mathrm{mult}_{\chi}(E_j(h))}{\dim \E_j(h)\cdot [ \pi_{\chi}|_{H}:\mathds{1}]}\,=\,\vol_{d\widetilde{\Sigma}_c}\widetilde{\Sigma}_c\; +\;\mathrm{O}\Big(h^{\delta}+h^{\frac{1-(2\kappa+3)\vartheta}{2\kappa +4}-\delta}\left (\log h^{-1}\right)^{{\Lambda^{}}-1}\Big).\label{eq:weyl1formula}
\eqn
\end{thm}
\vspace*{-1em}
\qed\\
Again, in the special case that $\Wh=\{\chi\}$ for all $h\in (0,1]$ and some fixed $\chi\in \widehat G$ we obtain
\begin{thm}[\bf Equivariant Weyl law for single isotypic components]
\label{thm:weyl41}Choose a fixed $\chi\in \widehat G$. Then one has in the limit $h\to 0$
\bqn
(2\pi)^{n-\kappa} h^{n-\kappa-\delta}\hspace{-1.5em}\sum_{\begin{array}{c}\scriptstyle 
j\in\mathbb{N}:\\ \scriptstyle E_{j}(h)\in\, [c,c+h^\delta]\end{array}}\hspace{-1em} \frac{\mathrm{mult}_{\chi}(E_j(h))}{\dim \E_j(h)}\,=\,[ \pi_{\chi}|_{H}:\mathds{1}]\vol_{d\widetilde{\Sigma}_c}\widetilde{\Sigma}_c\; +\;\mathrm{O}\Big(h^{\delta}+h^{\frac{1}{2\kappa +4}-\delta}\left (\log h^{-1}\right)^{{\Lambda^{}}-1}\Big).\label{eq:weyl1formula2}
\eqn
\end{thm}
\vspace*{-1em}
\qed
\begin{rem}
Note that the leading terms in the formulas above are non-zero. Indeed, if $c$ is a regular value of $p$ and consequently of $\widetilde{p}$, both $\Sigma_c$ and $\widetilde{\Sigma}_{c}$ are non-degenerate hypersurfaces, which implies that their volumes are non-zero. 
\end{rem}

As mentioned before, the proof of  the generalized equivariant semiclassical Weyl law in Theorem \ref{thm:weyl2} relies on the singular equivariant asymptotics which are the content of  Theorem \ref{thm:main}. Hereby  one cannot assume that the considered integrands are supported away  from the singular part of $\Omega$, in particular when localizing to $\Sigma_c\cap \Omega_{\mathrm{reg}}$ in \eqref{eq:4.11}. This means that for general group actions a desingularization process is indeed necessary, as the following examples illustrate. 

\begin{examples}
\label{ex:4.4}
\begin{enumerate}[leftmargin=*]\quad \item Let $G$ be a compact Lie group of dimension at least $1$, acting effectively and locally smoothly on the $n$-sphere $S^n$ with precisely one orbit type. Then $G$ either acts transitively or freely on $S^n$ \cite[Theorem IV.6.2]{bredon}. In the latter case, $G$ is either $S^1$, $S^3$, or the normalizer of $S^1$ in $S^3$. Consequently, if  $M$ is an arbitrary compact $G$-manifold, 
$S^\ast M$ will contain non-principal isotropy types in general. As a simple example, consider the linear action of $G=S^1$ on the $3$-sphere $M=S^3=\mklm{x \in \R^{4}: \norm{x}=1}$ given by
\bqn 
S^1=\mklm{z=e^{i\phi}, \, \phi \in [0,2\pi) } \ni  z \longmapsto R(z)=
\left (
\begin{matrix}
 1 & 0 & 0 & 0 \\ 0 & 1 & 0 & 0 \\ 0 & 0 & \cos \phi & \sin \phi \\ 0 & 0 & -\sin \phi & \cos \phi 
\end{matrix}
\right )  \in \SO(4)
\eqn
with isotropy types  $(\mklm{e})$ and $(S^1)$. The induced action on the tangent bundle  $TS^3=\coprod_x T_xS^3=\coprod_x \mklm{(x,v)\in \R^8: x\in S^3,\;v \perp x}$, which we identify with $T^\ast S^3$ via the induced metric, is given by $z \cdot (x,v) \mapsto (R(z) x,R(z) v)$, and has the same isotropy types. Let now $x\in S^3(S^1)$ be of singular orbit type. Then $x_3=x_4=0$ and $S^1$ acts on
\bqn 
S_xS^3=\mklm{(x,v)\in S^3\times S^3: v_1x_1+v_2x_2=0}
\eqn
with isotropy types  $(\mklm{e})$ and $(S^1)$. In particular, the $S^1$-action on $S_xS^3\simeq S_x^\ast S^3$ is neither transitive nor free.

\item Let $M=G$ be a Lie group with a left-invariant Riemannian metric and  $K\subset G$ a compact subgroup. Consider the left action of $G$ on itself  and the decomposition of $T^\ast G$    into isotropy types with respect to  the induced left $K$-action.  Taking into account the left trivialization $T^\ast G\simeq G \times \g^\ast $ explained in \cite[Example 4.5.5]{ortega-ratiu} one has 
\bqn 
 (S^\ast G)(H)=S^\ast(G(H))
 \eqn
 for an arbitrary closed subgroup $H\subset K$.  Thus, in general, the co-sphere bundle of $G$ will contain non-principal isotropy types. Assume now that $G$  is compact and consider a Schr\"odinger operator $P(h)$ on $G$ with $K$-invariant symbol function $p$. Let $c\in \R$ be  a regular value of $p$ and $\Sigma_c=p^{-1}(\mklm{c})$. Then the results of Theorem \ref{thm:weyl2} apply. By the previous considerations, 
\bqn 
{S^\ast G\cap \Omega_{\mathrm{reg}}}={\Omega \cap S^\ast (G(H))}, \qquad \Omega=\Jbb_K^{-1}(0),
\eqn
where $\Jbb_K: T^\ast G \to \k^\ast$ is the momentum map of the $K$-action, and $H$  a principal isotropy group. Consequently, the closure of ${S^\ast G\cap \Omega_{\mathrm{reg}}}$, and more generally of ${\Sigma_c\cap \Omega_{\mathrm{reg}}}$, will contain non-principal isotropy types in general. 

\end{enumerate}

\end{examples}

In case that $G$ acts on $M$ with finite isotropy groups, $G$-invariant pseudodifferential operators on $M$ correspond to pseudodifferential operators on the orbifold $\widetilde M$, and viceversa. In fact, the spectral theory of elliptic operators on compact orbifolds has attracted quite much attention  recently \cite{dryden-et-al,stanhope-uribe,kordyukov12} and, as mentioned at the beginning of this paper, our work can be viewed  as part of an attempt to develop a  spectral theory of elliptic operators on  general singular $G$-spaces. \\

\appendix

\section{}

In this appendix, we shall  collect a few useful technical facts related to the spaces and measures introduced in Section \ref{subsec:spaces}. We will refer to this summary also in Part II \cite{kuester-ramacher15b}. As before, we are not assuming that the considered measures are normalized, unless otherwise stated. With the notation as in Section \ref{subsec:spaces} we have

\begin{lem}\label{lem:pushfwd} The measure $d\widetilde{\Omega}_{\textrm{reg}}$ agrees with the Riemannian volume density defined by the Riemannian metric on $\widetilde{\Omega}_{\textrm{reg}}$ that is induced by the Sasaki metric on $T^*M$. 
\end{lem}
\begin{proof}
By \cite[Theorem 4.6]{blair} all metrics on $\widetilde{\Omega}_{\textrm{reg}}$ which are associated to the symplectic form $\widetilde{\omega}$ by an almost complex structure define the same Riemannian volume density, and that density agrees with the one defined by the symplectic form $\widetilde{\omega}$. Hence, it suffices to show that the Riemannian metric on $\widetilde{\Omega}_{\textrm{reg}}$ induced by the $G$-invariant Sasaki metric on $T^*M$ is associated to $\widetilde{\omega}$ by an almost complex structure. Now, the Sasaki metric $g_S$ on $T^*M$ is associated to the canonical symplectic form $\omega$ on $T^*M$ by an almost complex structure $\mathcal{J}:TT^*M\to TT^*M$. Consequently, the Riemannian metric $\iota^*g_S$ on $\Omega_{\textrm{reg}}$ is associated to the symplectic form $\iota^*\omega$ by the almost complex structure $\iota^*\mathcal{J}$, where $\iota:\Omega_{\textrm{reg}}\to T^*M$ is the inclusion. Since both $\iota^*g_S$ and $\iota^*\omega$ are $G$-invariant, $\iota^*\mathcal{J}:T\Omega_{\textrm{reg}}\to T\Omega_{\textrm{reg}}$ is $G$-equivariant, and therefore induces an almost complex structure $\widetilde{\iota^*\mathcal{J}}:T\widetilde{\Omega}_{\textrm{reg}}\to T\widetilde{\Omega}_{\textrm{reg}}$ which associates  the metric induced by ${\iota^*g_S}$ on $\widetilde{\Omega}_{\textrm{reg}}$ with $\widetilde{\omega}$.
\end{proof}

\begin{lem}\label{lem:nullset}$M-M_\text{reg}$ is a null set in $(M,dM)$, and $\Omega_\textrm{reg}- (T^*M_\textrm{reg}\cap \Omega)$ is a null set in $(\Omega_\textrm{reg},d\Omega_\textrm{reg})$.
\end{lem}
\begin{proof}The proof is completely analogous to the proof of \cite[{Lemma 3}]{cassanas-ramacher09}, and we refer the reader to \cite{kuester} for details.
\end{proof}

Similarly, on $\widetilde M=M/G$ we have 

\begin{cor}\label{cor:nullset2}
$\widetilde M-\widetilde{M}_\textrm{reg}$ is a null set in $(\widetilde M,d{\widetilde M})$, and $\widetilde{\Omega}_\textrm{reg}- (T^*M_\textrm{reg}\cap \Omega)/G$ is a null set in $(\widetilde{\Omega}_\textrm{reg},d\widetilde{\Omega}_\textrm{reg})$.
\end{cor}
\begin{proof}The first claim is true by definition of the measure $d{\widetilde M}$ and Lemma \ref{lem:nullset}. Concerning the second claim, note  that 
\[
(\Omega_\textrm{reg}- (T^*M_\textrm{reg}\cap \Omega_\textrm{reg}) )/G=\widetilde{\Omega}_\textrm{reg}- (T^*M_\textrm{reg}\cap \Omega_\text{reg})/G.
\]
Consequently,  (\ref{eq:prop81}) and Lemma \ref{lem:nullset} together yield
\[
\text{vol }\left(\widetilde{\Omega}_\textrm{reg}- (T^*M_\textrm{reg}\cap \Omega_\text{reg})/G\right)=\intop_{\Omega_\textrm{reg}- (T^*M_\textrm{reg}\cap \Omega_\text{reg})}\frac{1}{\text{vol}\left(G\cdot\eta\right)}d\Omega_\textrm{reg}(\eta)=0. 
\]
\end{proof}
\begin{lem}\label{lem:volsmooth}The orbit volume function $\vol_\O|_{M_\text{reg}}:M_\text{reg}\to\R$, $x\mapsto {\vol}(G\cdot x)$, is smooth. Moreover, if the dimension of the principal orbits is at least $1$, the function $\text{vol}\,_{\mathcal{O}}|_{M_\text{reg}}$ can be extended by zero to a continuous function $\overline{\text{vol}}\,_{\mathcal{O}}:M\to\R$.
\end{lem}
\begin{proof}See \cite[{Proposition 1}]{pacini}.\end{proof}

\begin{rem}\label{rem:vol}The function $\overline{\text{vol}}\,_{\mathcal{O}}:M\to\R$ from the previous lemma is in general different from the original orbit volume function $\text{vol}\,_{\mathcal{O}}:M\to\R$,  $x\mapsto \text{vol}(G\cdot x)$.  The latter function is by definition nowhere zero and not continuous if there are some orbits of dimension $0$ and some of dimension $>0$.
\end{rem}

\begin{lem}\label{lem:qnormmeas} On $\widetilde M=M/G$ we have \[
\frac{d{\widetilde M}}{\textrm{vol}}\Big|_{\widetilde{M}_{\text{reg}}}=d\widetilde{M}_{\text{reg}},\qquad\frac{d{\widetilde M}}{\textrm{vol}}\Big|_{\widetilde M-\widetilde{M}_{\text{reg}}}\equiv 0.\]
\end{lem}
\begin{proof}Considering (\ref{eq:prop80}), \eqref{eq:orbithaar}, and Corollary \ref{cor:nullset2}, the claimed relations are obvious.\end{proof}

\begin{cor}\label{cor:measures}The following two measures on $(T^*M_\text{reg}\cap\Omega_\text{reg})/G$ agree:
\begin{enumerate}
\item the measure $j^*d\widetilde{\Omega}_{\text{reg}}$, where $j$ is the inclusion $j:(T^*M_\text{reg}\cap\Omega_\text{reg})/G\hookrightarrow \widetilde{\Omega}_\text{reg}$ and $d\widetilde{\Omega}_{\text{reg}}$  the symplectic volume form on $\widetilde{\Omega}_{\text{reg}}$;
\item the measure $\Phi^*d(T^*\widetilde{M}_\text{reg})$, where $\Phi:(T^*M_\text{reg}\cap\Omega_\text{reg})/G \to T^*\widetilde{M}_\text{reg}$ is the canonical symplectomorphism from Lemma \ref{lem:isomorphic} and $d(T^*\widetilde{M}_\text{reg})$ the symplectic volume form on $T^*\widetilde{M}_\text{reg}$.
\end{enumerate}
\end{cor}

\begin{proof}The measures $d\widetilde{\Omega}_{\text{reg}}$ and  $d(T^*\widetilde{M}_\text{reg})$ are defined by the volume forms $\widetilde{\omega}^{n-\kappa}/(n-\kappa)!$ and $\widehat{\omega}^{n-\kappa}/(n-\kappa)!$, respectively, which implies that the measures $j^*d\widetilde{\Omega}_{\text{reg}}$ and $\Phi^*d(T^*\widetilde{M}_\text{reg})$ are defined by the volume forms $j^*\widetilde{\omega}^{n-\kappa}/(n-\kappa)!$  and  $\Phi^*\widehat{\omega}^{n-\kappa}/(n-\kappa)!$, respectively. Using compatibility of the wedge product with pullbacks and Lemma \ref{lem:isomorphic} we obtain
\[
j^*\widetilde{\omega}^{n-\kappa}=(j^*\widetilde{\omega})^{n-\kappa}=(\Phi^*\widehat{\omega})^{n-\kappa}=\Phi^*(\widehat{\omega}^{n-\kappa}).
\]
\end{proof}
The next lemma describes the hypersurface measures that we use frequently.
\begin{lem}\label{lem:hypersurf0}
Let $\Xbf$ be a smooth manifold equipped with a measure $d\Xbf$ given by a smooth volume density. Let $H\subset \Xbf$ be a smooth hypersurface. Then there is a canonical induced hypersurface measure $dH$ on $H$, given by the restriction of the volume density on $\Xbf$ to $H$. In case that $\Xbf$ is a Riemannian manifold and $d\Xbf$ is defined by the Riemannian volume density, then $dH$ is defined by the Riemannian volume density associated to the induced Riemannian metric on $H$.
\end{lem}
\begin{proof}
Let $\{U_\alpha,\gamma_\alpha\}_{\alpha\in \mathcal{A}}$ be an atlas for $\Xbf$ and $\{\varphi_\alpha\}$ a locally finite partition of unity subordinate to $\{U_\alpha\}$. Let $n:=\dim\Xbf$. For each $\alpha$, we choose the local coordinates $(x^\alpha_1,\ldots,x^\alpha_n)$ over $U_\alpha$ such that $H$ is described by the vanishing of the last coordinate, i.e.\ $H\cap U_\alpha=\{(x^\alpha_1,\ldots,x^\alpha_{n-1},0)\}$. By assumption,
\[
\int_{\Xbf}f\d \Xbf= \sum_{\alpha\in\mathcal{A}}\int_{U_\alpha}f\varphi_\alpha V_\alpha^{\Xbf}|\d x_1\wedge\d x_2\wedge\ldots\wedge\d x_n|\qquad \forall\;f\in C(\Xbf)
\]
with a collection of smooth functions $V^{\Xbf}_\alpha:U_\alpha\to [0,\infty)$, the integrals representing Lebesgue integrals over $\R^n$. Now, we can set for $f\in C(H)$
\[
\int_{H}f\d H:= \sum_{\alpha\in\mathcal{A}}\int_{U_\alpha\cap H}f(x_1,\ldots,x_{n-1})\varphi_\alpha(x_1,\ldots,x_{n-1},0) V_\alpha^{\Xbf}(x^\alpha_1,\ldots,x^\alpha_{n-1},0)|\d x_1\wedge\d x_2\wedge\ldots\wedge\d x_{n-1}|
\]
to define the desired smooth measure $dH$. It is easy to check that this definition is independent of the choice of the first $n-1$ coordinates on $\Xbf$, hence of the choice of coordinates on $H$, and the partition of unity. The last assertion of the theorem follows immediately by choosing  $V_\alpha^{\Xbf}=\sqrt{\text{det}(g_\alpha)}$, where $g_\alpha$ is the matrix representing the Riemannian metric on $\Xbf$ over the chart $U_\alpha$.
\end{proof}

The following technical lemma is a direct consequence of the Fubini theorem. 
\begin{lem}\label{lem:hypersurf} Let $c\in \R$ be a regular value of our Hamiltonian function $p$. For each $\delta>0$, let $I_\delta \subset [c-\delta,c+\delta]$ be a non-empty interval. Then, for all $f\in C(T^*M)$ the limit
\begin{equation}
\lim_{\delta\to 0}\frac{1}{\vol_\R(I_\delta)}\intop_{p^{-1}(I_\delta)}f\d(T^*M)=:\intop_{p^{-1}(\{c\})}f\d\Sigma_c\label{eq:defhypmeas}
\end{equation}
exists, and uniquely defines a finite measure $d\Sigma_c$ on the hypersurface $\Sigma_c:=p^{-1}(\{c\})$ that agrees with the hypersurface measure on $\Sigma_c$ induced by the measure $d(T^*M)$ on $T^*M$ (see Lemma \ref{lem:hypersurf0}). Furthermore, for each $f\in C(T^*M)$, one has the estimate in the limit $\delta\to0$
\begin{equation}
\frac{1}{\vol_\R(I_\delta)}\intop_{p^{-1}(I_\delta)}f\d (T^*M)-\intop_{p^{-1}(\{c\})}f\d\Sigma_c=\mathrm{O}(\delta).\label{eq:hypersurf}
\end{equation}
\end{lem}
\begin{proof}Since $M$ is compact, $p^{-1}([c-r,c+r])\subset T^*M$ is compact for every $r>0$. Thus, we can find $\varepsilon>0$ such that each $t\in [c-\varepsilon,c+\varepsilon]$ is a regular value of $p$. This implies that there is an atlas for $T^*M$ such that the intersection of any chart with $p^{-1}(\{t\})$ is given by the points whose last coordinate is equal to $t-c$ for each $t\in [c-\varepsilon,c+\varepsilon]$. As $p^{-1}([c-\varepsilon,c+\varepsilon])$ is compact, we can reduce such an atlas to a finite collection of charts that still cover $p^{-1}([c-\varepsilon,c+\varepsilon])$. Denote the so obtained finite collection of charts by $\{U_\alpha,\gamma_\alpha\}_{\alpha\in\mathcal{A}}$, $\gamma_\alpha:U_\alpha\to V_\alpha\subset \R^{2n}$, $U_\alpha\subset T^*M$. W.l.o.g.\ we can assume that $V_\alpha\subset \R^n$ is bounded. Let $\{\varphi_\alpha\}_{\alpha\in\mathcal{A}}$ be a partition of unity subordinated to the family $\{U_\alpha\}_{\alpha\in\mathcal{A}}$. Then, by definition, it holds for $\delta<\varepsilon$ and an interval $I_\delta\subset[c-\delta,c+\delta]$:
\[
\int_{p^{-1}(I_\delta)}f\d(T^*M) =\sum_{\alpha\in\mathcal{A}}\int_{\gamma_\alpha(U_\alpha\cap p^{-1}(I_\delta))} (f\cdot\varphi_\alpha)\big(\gamma_\alpha^{-1}(y)\big)\sqrt{\text{det}(g_\alpha(y))}\d y\qquad \forall\;f\in C(T^*M),
\]
where $g_\alpha$ is the local matrix defined by the Riemannian metric (the Sasaki metric) on $T^*M$ over the chart $U_\alpha$.  Due to our special choice of coordinates in the chart $U_\alpha$, we get for $f\in C(T^*M)$:
\begin{align*}
&\int_{p^{-1}(I_\delta)}f\d(T^*M) =\sum_{\alpha\in\mathcal{A}}\int_{V_\alpha} (f\cdot\varphi_\alpha)\big(\gamma_\alpha^{-1}(y_1,\ldots,y_{2n})\big)\sqrt{\text{det}(g_\alpha(y_1,\ldots,y_{2n}))}\d y_1\ldots\d y_{2n-1}\d y_{2n}\\
&=\sum_{\alpha\in\mathcal{A}}\int_{I_\delta}\int_{V_\alpha\cap\R^{2n-1}} (f\cdot\varphi_\alpha)\big(\gamma_\alpha^{-1}(y_1,\ldots,y_{2n-1},t)\big)\sqrt{\text{det}(g_\alpha(y_1,\ldots,y_{2n-1},t))}\d y_1\ldots\d y_{2n-1}\d t,
\end{align*}
where we used the Fubini theorem for the Lebesgue integral. Since $p^{-1}([c-\varepsilon,c+\varepsilon])$ is compact, we can assume without loss of generality that the coefficients (and hence the determinant) of $g_\alpha$ and all their partial derivatives are bounded. Furthermore, we know that the functions $f$ and $\varphi_\alpha$ are uniformly continuous on $p^{-1}([c-\varepsilon,c+\varepsilon])$. This implies for each $\alpha\in\mathcal{A}$ and $y\in V_\alpha$:
\begin{multline*}
\Big|(f\cdot\varphi_\alpha)\big(\gamma_\alpha^{-1}(y_1,\ldots,y_{2n-1},t)\big)\sqrt{\text{det}(g_\alpha(y_1,\ldots,y_{2n-1},t))}\\
-(f\cdot\varphi_\alpha)\big(\gamma_\alpha^{-1}(y_1,\ldots,y_{2n-1},c)\big)\sqrt{\text{det}(g_\alpha(y_1,\ldots,y_{2n-1},c))}\Big|\leq C_\alpha(|t-c|)\qquad\forall\;t\in I_\delta,
\end{multline*}
with some constant $C_\alpha>0$ that is independent of $y$. To shorten the notation, set
\[
\mathcal{I}(c,\alpha):=\intop_{V_\alpha\cap\R^{2n-1}} (f\cdot\varphi_\alpha)\big(\gamma_\alpha^{-1}(y_1,\ldots,y_{2n-1},c)\big)\sqrt{\text{det}(g_\alpha(y_1,\ldots,y_{2n-1},c))}\d y_1\ldots\d y_{2n-1}
\]
With $|t-c|\leq \delta$ we then get
\begin{align*}
&\int_{I_\delta}\int_{V_\alpha\cap\R^{2n-1}} (f\cdot\varphi_\alpha)\big(\gamma_\alpha^{-1}(y_1,\ldots,y_{2n-1},t)\big)\sqrt{\text{det}(g_\alpha(y_1,\ldots,y_{2n-1},t))}\d y_1\ldots\d y_{2n-1}\d t\\
&=\intop_{I_\delta}\big(\mathcal{I}(c,\alpha) +\mathrm{O}_\alpha(|t-c|)|\big)\d t
=\vol_\R(I_\delta)\mathcal{I}(c,\alpha) + \int_{I_\delta}\mathrm{O}_\alpha({|t-c|})|\d t
=\vol_\R(I_\delta)\big(\mathcal{I}(c,\alpha)+\mathrm{O}_\alpha(\delta)\big).
\end{align*}
Since $\mathcal{A}$ is finite, we conclude
\bqn
\frac{1}{\vol_\R(I_\delta)}\int_{p^{-1}(I_\delta)}f\d(T^*M) 
=\sum_{\alpha\in\mathcal{A}}\mathcal{I}(c,\alpha)+\mathrm{O}_\alpha(\delta)
=\intop_{p^{-1}(\{c\})}f\d\mu_c+ \mathrm{O}(\delta).
\eqn
\end{proof}
From the previous lemma, one immediately obtains a symmetry-reduced version for $T^*\widetilde{M}_\text{reg}$. 
\begin{cor}\label{cor:hypersurf2}Let $\widetilde p:T^*\widetilde M_\text{reg}\to \R$ be the function induced by our Hamiltonian function $p$ via Lemma \ref{lem:isomorphic}. Let $c\in \R$ be a regular value of $\widetilde{p}$. For each $\delta>0$, let $I_\delta \subset [c-\delta,c+\delta]$ be a non-empty interval. Then, for all $G$-invariant $f\in C(T^*M)$, inducing $\widetilde f\in C(T^*\widetilde M_\text{reg})$, the limit
\begin{equation}
\lim_{\delta\to 0}\frac{1}{\vol_\R(I_\delta)}\intop_{\widetilde p^{-1}(I_\delta)}\widetilde f\d(T^*\widetilde M_\text{reg})=:\intop_{\widetilde p^{-1}(\{c\})}\widetilde f\d\widetilde \Sigma_c\label{eq:defhypmeas2}
\end{equation}
exists, and uniquely defines a finite measure $d\widetilde \Sigma_c$ on the hypersurface $\widetilde \Sigma_c:=\widetilde p^{-1}(\{c\})$ that agrees with the measure induced from the canonical volume density on $T^*\widetilde M_\text{reg}$. Furthermore, for each $G$-invariant $f\in C(T^*M)$, one has the estimate in the limit $\delta\to0$
\begin{equation}
\frac{1}{\vol_\R(I_\delta)}\intop_{\widetilde p^{-1}(I_\delta)}\widetilde f\d (T^*\widetilde M_\text{reg})-\intop_{\widetilde p^{-1}(\{c\})}\widetilde f\d\widetilde\Sigma_c=\mathrm{O}(\delta).\label{eq:hypersurf2}
\end{equation}
\end{cor}
\begin{proof}
Since the functions $\widetilde p$, $\widetilde f$, and the Riemannian metric on $T^*\widetilde M_\text{reg}$ are all induced by $G$-invariance from their counterparts on $T^*M$, we can proceed by analogy to the proof of Lemma \ref{lem:hypersurf}, even though $\widetilde M_\text{reg}$ is not compact. In particular, we can find $\varepsilon>0$ such that each $t\in [c-\varepsilon,c+\varepsilon]$ is a regular value of $\widetilde p$, we can cover $\widetilde p^{-1}([c-\varepsilon,c+\varepsilon])$ by finitely many charts $\{U_\alpha,\gamma_\alpha\}_{\alpha\in\mathcal{A}}$ with a subordinated partition of unity $\{\varphi_\alpha\}_{\alpha\in \mathcal{A}}$, and $\widetilde f$ and each $\varphi_\alpha$ are uniformly continuous on $\widetilde p^{-1}([c-\varepsilon,c+\varepsilon])$. Furthermore, the local coefficients and their derivatives of the Riemannian metric on $T^*\widetilde M_\text{reg}$ are bounded. Thus, we can perform analogous calculations and estimates as in the proof of Lemma \ref{lem:hypersurf}.
\end{proof}
\begin{rem}\label{rem:lastrem}
When $V\equiv 0$ and $c=1$, the hypersurface measure obtained in Corollary \ref{cor:hypersurf2} agrees with the Liouville measure $d(S^*\widetilde M_{\text{reg}})$, since the Riemannian metric on $\widetilde M_{\text{reg}}$ is induced by the $G$-invariant Riemannian metric on $M$, and the volume densities on $T^*\widetilde M_{\text{reg}}$ induced by its canonical symplectic form and the natural Sasaki metric agree.
\end{rem}

\providecommand{\bysame}{\leavevmode\hbox to3em{\hrulefill}\thinspace}
\providecommand{\MR}{\relax\ifhmode\unskip\space\fi MR }
\providecommand{\MRhref}[2]{%
  \href{http://www.ams.org/mathscinet-getitem?mr=#1}{#2}
}
\providecommand{\href}[2]{#2}


\end{document}